\documentclass[12pt]{amsart}
\usepackage[T1]{fontenc}

\usepackage{amsmath,amsfonts,amssymb,amsthm}
\usepackage{enumerate,bbm,physics,mathtools,esint}
\usepackage{tikz,graphicx,color}
\usepackage{caption,float,xcolor}
\usepackage{fullpage}
\usepackage{hyperref}
\usepackage{comment}

\newtheorem{theorem}{Theorem}[section]
\newtheorem{lemma}[theorem]{Lemma}

\theoremstyle{definition}
\newtheorem{definition}[theorem]{Definition}
\newtheorem{example}[theorem]{Example}

\theoremstyle{remark}
\newtheorem{remark}[theorem]{Remark}

\numberwithin{equation}{section}

\newcommand{\R}{{\mathbb R}}

\newcommand{\Z}{{\mathbb Z}}
\newcommand{\N}{{\mathbb N}}
\newcommand{\E}{{\mathbb E}}
\newcommand{\Lam}{\Lambda}
\newcommand{\nuMangoldt}{\nu_{\Lambda}}

\newcommand{\eps}{{\epsilon}}

\newcommand{\Erdos}{{Erd\H{o}s}}
\newcommand{\Sarkozy}{S\'{a}rk\"{o}zy} 
\newcommand{\Szemeredi}{Szemer\'edi}

\begin{document}

\title[Primitive sets and von Mangoldt chains]{Primitive sets and von Mangoldt chains:\\ Erd\H{o}s Problem \#1196 and beyond}

\author{Boris Alexeev}
\address{OpenAI, San Francisco, CA 94158}
\email{boris.alexeev@gmail.com}

\author{Kevin Barreto}
\address{Queens' College, University of Cambridge, Cambridge, CB3 9ET, United Kingdom}
\email{kb799@cam.ac.uk}

\author{Yanyang Li}
\address{School of Mathematics, Southeast University, Nanjing 211189, P.~R.~China}
\email{liyanyang1219@gmail.com}

\author{Jared Duker Lichtman}
\address{
Department of Mathematics \\
450 Jane Stanford Way, Stanford, CA 94305-2125, USA
}
\email{jared.d.lichtman@gmail.com}

\author{Liam Price}
\email{liam.price2002@gmail.com}

\author{Jibran Iqbal Shah}
\address{
Department of Mathematics \\
University of Toronto \\
Toronto, ON \\
Canada
}
\email{jibraniqbal.shah@mail.utoronto.ca}

\author{Quanyu Tang}
\address{School of Mathematics and Statistics, Xi'an Jiaotong University, Xi'an 710049, P.~R.~China}
\email{tangquanyu827@gmail.com}

\author{Terence Tao}
\address{Department of Mathematics \\
University of California, Los Angeles \\
Los Angeles, CA 90095 \\
USA}
\email{tao@math.ucla.edu}

\begin{abstract}
A set of integers is primitive if no number in the set divides another. We introduce a new method for bounding \Erdos{} sums of primitive sets, suggested from output of GPT-5.4 Pro, based on Markov chains with von Mangoldt weights. The method leads to a host of applications, yet seems to have been overlooked by the prior literature since \Erdos{}' seminal 1935 paper.

As applications, we prove two 1966 conjectures of \Erdos{}--\Sarkozy{}--\Szemeredi{}, on primitive sets of large numbers (\#1196) and on divisibility chains (\#1217). The method also provides a short proof of the \Erdos{} Primitive Set Conjecture (\#164), as well as the related claim that $2$ is an ``\Erdos{}-strong'' prime. Moreover, the method resolves (a revised form of) the Banks--Martin conjecture, which has long been viewed as a unifying `master theorem' for the area.
\end{abstract}
\maketitle

\section{Introduction}

The \emph{divisibility poset} $(\N,|)$ is the set of natural numbers $\N = \{1,2,\dots\}$ equipped with the divisibility relation $a|b$ as the partial ordering.  One can organize this poset into ``layers''
\begin{align*}
\N_k&\coloneq \{n:\Omega(n)=k\} \\
\N_{\geq k}&\coloneq \{n:\Omega(n) \geq k\} \\
\N_{>k}&\coloneq \{n:\Omega(n) >k\} 
\end{align*}
for $k \geq 0$, where $\Omega(n)$ denotes the number of prime factors of $n$, counting multiplicity; see Figure \ref{fig-divis}.  For instance,
$$ \N_1 = \{2,3,5,7,\dots\}$$
is the set of primes,
$$ \N_{>1} = \{4,6,8,9,10,\dots\}$$
is the set of composite numbers, and
$$ \N_{\geq 1} = \{2,3,4,5,\dots\}$$
is the set of natural numbers greater than or equal to $2$.  Recall that a set $A \subset \N$ is called \emph{primitive} if it is an antichain in the divisibility poset, or equivalently, if no two distinct elements of $A$ divide one another.  Thus, for instance, the primes $\N_1$ are primitive, as well as $\N_k$ for every $k \geq 0$.

The class of primitive sets is quite general with wide ranging examples, from $\Z\cap (x,2x]$ integers in a dyadic interval to the set of perfect numbers. Recall since Ancient Greece, a number $n$ is as `perfect,' `abundant,' or `deficient,' depending on if the sum of its proper divisors equals $n$, is greater than $n$, or is less than $n$, respectively.

The study of primitive sets emerged in the 1930s as a generalization of studying perfect and abundant numbers. A classical theorem of Davenport asserts that the set of abundant numbers has a positive asymptotic density. This was originally proved by technical analytic methods, but Erd\H{o}s soon found an elementary proof by using primitive abundant numbers.\footnote{More precisely, `primitive non-deficient numbers'} The proof methods led people to introduce the abstract definition of primitive sets and study them for their own sake. See Hall \cite{Hsetmult} or Halberstam--Roth \cite[\S 5]{HalbRoth} for detailed introductions to the subject.

\begin{figure}
  \centering
  \includegraphics[width=0.75\textwidth]{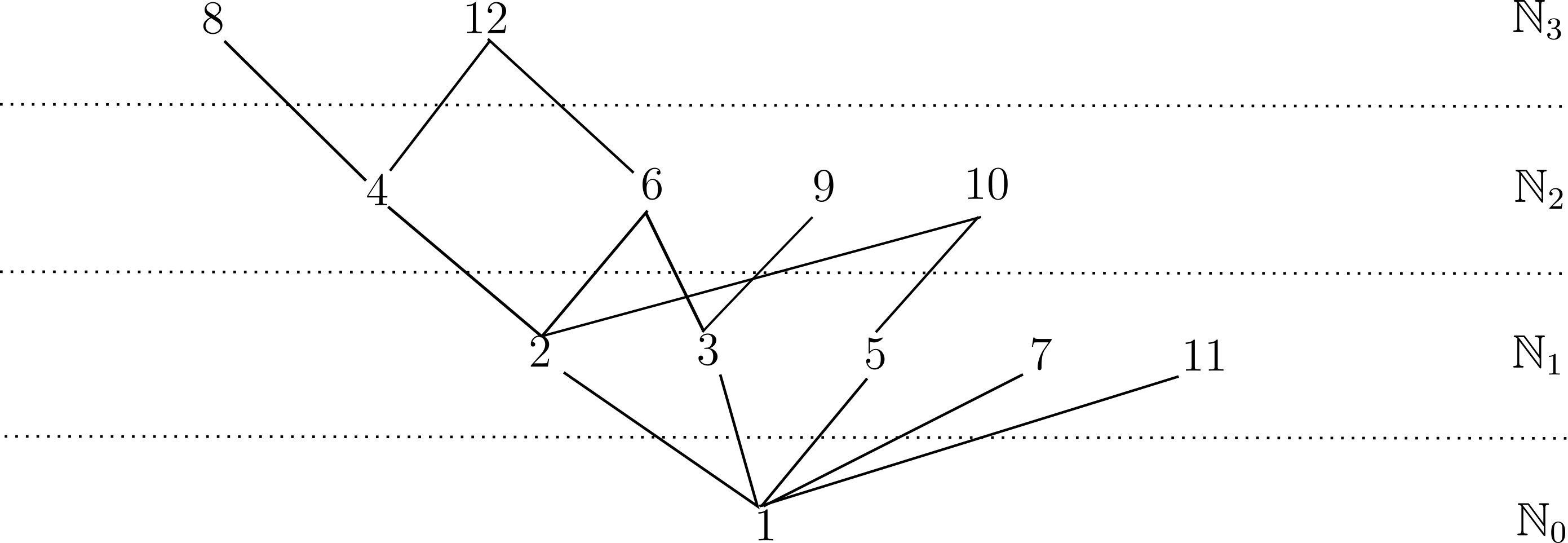}
  \caption{A portion of the divisibility poset on $\N$.}
  \label{fig-divis}
\end{figure}

If we exclude the degenerate case $A = \{1\}$, it is natural to measure the size of such sets by the doubly logarithmic size
$$f(A)\coloneq \sum_{a\in A} \nu_0(a)$$
where $\nu_0 \colon \N_{\geq 1} \to [0,+\infty)$ is the doubly harmonic weight
\begin{equation}\label{nu-def} \nu_0(a) \coloneq \dv{a}(\log\log a) = \frac{1}{a\log a}.
\end{equation}
In Remark \ref{invariants} below, we introduce an asymptotically equivalent weight $\nuMangoldt$ to $\nu_0$ which is more naturally adapted to the study of primitive sets; however, our focus here shall mostly be on the original doubly harmonic weight $\nu_0$, as this is the choice of weight that traditionally has been studied in the literature.

In this article, we introduce a Markov chain-based method on the divisibility poset $(\N,|)$ to study the size of primitive sets, reproving several previous results in the literature and resolving some open conjectures.  Several of these results were established with various levels of AI assistance; see Section \ref{acknowledgments-sec}.

We now state our results and provide some (incomplete) references to the literature; we refer the reader to these papers for further discussion and a history of prior partial results.

Our first theorem, proven in Section \ref{1196-sec}, is as follows.

\begin{theorem}[\Erdos{}--S\'ark\"ozy--Szemer\'edi, \#1196] \label{conj:1196}
If $A$ is a primitive set contained in $[x,\infty)$ for some $x \geq 2$, then
$$f(A) \leq 1 + O\left(\frac{1}{\log x}\right).$$
\end{theorem}

\Erdos{}, S\'ark\"ozy, and Szemer\'edi \cite[(20)]{ess2}, \cite[p.~101]{erdos-survey}, \cite[p. 244, Problem 2.2]{sarkozy1}, \cite[p. 224, Problem 2]{sarkozy2},
\cite[Problem \#1196]{bloom} had conjectured that $f(A) \leq 1+o(1)$ as $x \to \infty$; Theorem~\ref{conj:1196} proves this conjecture with a more quantitative error term.  The previous best-known upper bound was
$$ f(A) \leq e^\gamma \frac{\pi}{4} + o(1) \approx 1.399,$$
due to the fourth author \cite{Lichtman}, building upon previous work in \cite{ezhang}, \cite{zhang1}, \cite{zhang2}.  In the opposite direction, it is known \cite{glw} that
\begin{equation}\label{nk}
f(\N_k) = 1 - (c+o(1)) k^2 2^{-k}
\end{equation}
for an explicit constant $c \approx 0.0656$, so the bound in Theorem \ref{conj:1196} is sharp up to the lower-order term.  A version of our proof of Theorem \ref{conj:1196} has been formalized in Lean by Math Inc.~\cite{lean-erdos1196}.  Several additional proofs of this result are now known; see Remark \ref{1196-alt-rem}, Remark \ref{remark:qualitative-1196}, and Section \ref{discussion}.  

Next, in Section \ref{EPS-sec} we provide an alternative proof of the \Erdos{} primitive set conjecture.

\begin{theorem}[\Erdos{} primitive set conjecture, \#164] \label{conj:EPS}
For any primitive $A$, $f(A)\le f(\N_1) = 1.6366\dots$\footnote{The decimal expansion of $f(\N_1)$ is \href{https://oeis.org/A137245}{OEIS A137245}, with the initial digits worked out in \cite[pp.~208--209]{cohen}.}.
\end{theorem}

This result was conjectured in \cite{erdos76}, \cite[Conjecture 2.1]{erdos86}, \cite[1.12]{various}, \cite[Problem \#164]{bloom}.  A uniform bound $f(A) \ll 1$ was first established by \Erdos{} \cite{erdos35}.  This conjecture was first solved by the fourth author \cite{Lichtman}; we provide a shorter proof here.  A variant of this proof was formalized in Lean by the first author~\cite{lean-erdos164}.  We refer the reader to \cite{Lichtman} for a more detailed history of this conjecture.

As we shall show in Section \ref{obm-sec}, the method of proof also establishes a related conjecture of Banks and Martin:

\begin{theorem}[Odd Banks--Martin] \label{conj:oddBM}
Let $k\ge 1$ and suppose $A$ is a primitive set in $\N_{\geq k}$. Then for any set of odd primes $\mathcal Q$,
\begin{align*}
f(A(\mathcal Q)) \ \le\ f\big(\mathbb N_k(\mathcal Q)\big),
\end{align*}
where $A(\mathcal Q)$ denotes the set of members of $A$ composed of primes in $\mathcal Q$.
\end{theorem}

Without the requirement that the primes in $\mathcal Q$ are odd, this was conjectured in \cite{banks-martin}, and verified in special cases; however, it was observed in \cite{Lalmost} (see also \cite{Lichtman, glw}) that the claim is actually false if ${\mathcal Q}$ was permitted to contain $2$. The revised form above was proposed in \cite[Conj. 7.1]{Lichtman}.  As a corollary of this theorem, we see that $f(\mathbb N_k({\mathcal Q}))$ is a non-increasing function of $k$ whenever ${\mathcal Q}$ consists of odd primes. We also observe in Remark~\ref{remark:qualitative-1196} that Theorem~\ref{conj:oddBM} implies the qualitative form of Theorem~\ref{conj:1196}, namely the bound $f(A) \leq 1+o(1)$ as $x \to \infty$, thus yielding an alternate proof of the conjecture of \Erdos{}, \Sarkozy{} and \Szemeredi{} in its original form.

Following \cite{LP}, \cite{Lichtman}, we call a prime $p$ \emph{\Erdos{}-strong} if one has the inequality
$$f(A) \leq f(\{p\}) = \nu_0(p)$$ whenever $A$ is a primitive set contained in the set of natural numbers with least prime factor $p$.  In \cite{LP} it was observed that Theorem \ref{conj:EPS} would follow if all primes were \Erdos{}-strong.  This was verified for odd primes in \cite{Lichtman}, building on some partial results in \cite{LP}.  In Section \ref{2-strong-sec} we resolve the remaining case of the prime $2$, thus giving yet another proof of Theorem \ref{conj:EPS} and answering a question from \cite{Lichtman}.

\begin{theorem}\label{2-strong} $2$ is \Erdos{}-strong.
\end{theorem}

Our methods also give a short proof of an inequality of Ahlswede, Khachatrian, and S\'ark\"ozy, which we present in Section \ref{AKS-sec}:

\begin{theorem}\label{AKS-thm}  If $A$ is a primitive set, then
$$ \sum_{n \in A \cap [y/x,y]} \frac{1}{n} \ll \frac{\log x}{\sqrt{\log\log x}}$$
whenever $3 \leq x \leq y$.
\end{theorem}

This result was first established in \cite[Theorem 3]{AKS}; the $y=x$ case of this theorem, namely
\begin{equation}\label{ax}
\sum_{n \in A \cap [1,x]} \frac{1}{n} \ll \frac{\log x}{\sqrt{\log\log x}}
\end{equation}
is a classical result of Behrend \cite{behrend}.  In fact, the sharper bound
$$ \sum_{n \in A \cap [1,x]} \frac{1}{n} \leq (1 + o(1)) \frac{\log x}{\sqrt{2\pi \log\log x}}$$
as $x \to \infty$ is known (uniformly in $A$), which is best possible up to the $o(1)$ error; see \cite{ess1}, \cite{ess2}.  We were unable to recover this bound with our methods.  However, it should be possible to recover several other results from \cite{AKS} by these techniques; we leave this task to the interested reader.  Furthermore, our method naturally gives an improvement to Theorem~\ref{AKS-thm} of ``LYM type''; see Remark \ref{lym-rem}.

Finally, in Section \ref{divis-sec} we present a result on a separate problem of \Erdos{}, S\'ark\"ozy, and Szemer\'edi:

\begin{theorem}[\Erdos{}--S\'ark\"ozy--Szemer\'edi, \#1217] \label{conj:1217}
Let $A \subset \N$ be such that the upper doubly logarithmic density
$$ \Delta \coloneq \limsup_{x\to\infty}\frac{1}{\log\log x} f(A \cap [1,x])$$
is positive.  Then there exists a strictly increasing infinite divisibility chain
$$ n_0 | n_1 | n_2 | \dots$$
in $A$ such that
$$ \limsup_{x\to\infty}\frac{1}{\log\log x} \# \{ i: n_i \leq x \} \geq \Delta.$$
\end{theorem}

This result was conjectured in \cite{ess0}, under the stronger hypothesis that $A$ had positive lower logarithmic density.  A classical result of Davenport and \Erdos{} \cite{davenport} asserts that all sets of positive lower natural density contain at least one strictly increasing infinite divisibility chain.

\begin{remark}\label{GPT-rem-1} Let $h \geq 1$.  If a set $A \subset \N$ is ``$h$-primitive'' in the sense that all strictly increasing divisibility chains $n_1 | \dots | n_k$ in $A$ have length at most $h$, then by iteratively peeling off the minimal elements of $A$ in the divisibility poset, one may decompose $A$ into the union of at most $h$ primitive sets (cf., Figure \ref{fig-divis}).  Using the triangle inequality, one can then obtain analogues of the above results which lose a factor of $h$.  For instance, if we additionally have $A \subset [x,\infty)$, then from Theorem~\ref{conj:1196} we have 
\begin{equation}\label{fah}
f(A) \leq \left(1 + O\left(\frac{1}{\log x}\right)\right) h.
\end{equation}
\end{remark}

\subsection{Methods of proof}

The strategy of proof in all of the above results is to construct either an increasing random divisibility chain
\begin{equation}\label{div-up}
 n_0 | n_1 | n_2 | \dots
 \end{equation}
or a decreasing random divisibility chain
\begin{equation}\label{div-down}
\dots | n_2 | n_1 | n_0
\end{equation}
in the divisibility poset $(\N,|)$, and exploit the basic duality between primitive sets and divisibility chains, in that any primitive set can meet such a chain in at most one natural number $n$.  Such chains will be constructed via various upward or downward Markov chains on the divisibility poset. A particularly natural choice for a downward Markov chain \eqref{div-down} is provided by what we call the \emph{von Mangoldt downward chain}, in which a natural number $n \in \N_{\geq 1}$ transitions to a factor $n/q$ with probability $\Lambda(q)/\log n$, where $\Lambda$ is the von Mangoldt function.  This is a Markov chain by the basic identity
\begin{equation}\label{vmi}
\sum_{q|n} \Lambda(q) = \log n
\end{equation}
for all $n \geq 1$.  For future reference, we also recall the Dirichlet series transform
\begin{equation}\label{lam-s}
\sum_{q}\frac{\Lam(q)}{q^s}
=-\frac{\zeta'(s)}{\zeta(s)}
\end{equation}
of the identity \eqref{vmi}, valid for all\footnote{In this paper, we will not need to analytically continue any of these identities into the complex plane.} $s>1$.  Here, of course, $\zeta$ denotes the Riemann zeta function.  Many of the proofs of the results above will involve working with some minor modification of this von Mangoldt downward chain.

For several of the arguments, it is more convenient to introduce an \emph{adjoint} of such a downward chain with respect to a reference weight $\nu$, producing an upward divisibility chain \eqref{div-up} rather than a downward one.  In order for the adjoint construction to work, the 
weight $\nu$ needs to be invariant, or at least sub-invariant, with respect to the downward Markov chain; in the sub-invariant case, this requires one to adjoin an absorbing state $\infty$ to the Markov chain, though this state plays little role in the analysis.  It turns out that the doubly harmonic weight $\nu_0$ defined in \eqref{nu-def} will be sub-invariant against the von Mangoldt downward chain and several of its variants, which is the main ingredient needed to establish most of the above results. (We also construct a genuinely invariant weight $\nuMangoldt$ asymptotic to $\nu_0$ in Remark \ref{invariants}.)

Many of the arguments can be rephrased in the language of flow networks rather than Markov chains; see Section~\ref{subsection:flow} for an example.

\begin{remark}[Interpretation using chain/antichain duality] A result of Stanley \cite{stanley} asserts that the convex hull of the indicator functions $1_A$ of antichains in a poset ${\mathcal N}$ is given by the \emph{Stanley chain polytope} of weight functions $w \colon {\mathcal N} \to [0,+\infty)$ which obey the condition $\sum_j w(n_j) \leq 1$ for all strictly increasing chains $n_0, n_1, \dots$ in ${\mathcal N}$.  From this and linear programming duality (or the Farkas lemma), we see that if $\nu \colon {\mathcal N} \to [0,+\infty)$ is a weight and $M>0$, then an upper bound of the form $\sum_{n \in A} \nu(n) \leq M$ holds for all antichains $A \subset {\mathcal N}$ if and only if there is a probability distribution on chains such that each element $n$ of the poset lies in the chain with probability at least $\nu(\{n\})/M$.  This observation already helps explain why the strategy of this paper is a viable one; but a key additional insight is that in the context of the divisibility poset, taking the random chain to be \emph{Markovian} can be a highly efficient choice, especially if the chain is related to the von Mangoldt process. 
\end{remark}

\subsection{Notation}

We use $X \ll Y$, $Y \gg X$, or $X = O(Y)$ to denote the bound $|X| \leq CY$ for an absolute constant $C$. If the implied constant $C$ needs to depend on additional parameters, we indicate this via subscripts.

All sums over $p$ are understood to be over primes.  

If $E$ is a set, we use $\# E$ to denote its cardinality.  If $P$ is a statement or event, we use $1_P$ to denote its indicator; thus $1_P=1$ if $P$ is true and $1_P=0$ if $P$ is false.

If $A$ is a set of natural numbers and $c \in \N$, we use $c \cdot A \coloneq \{ c n: n \in A \}$ to denote the dilate of $A$ by $c$.

\section{Downward and upward Markov chains}

Our arguments will rely on constructing various downward and upward Markov chains, which lead to downward and upward divisibility chains, respectively.  In this section, we lay out our general notation for such chains.

\subsection{Downward chains}

We begin with our notation for downward chains.

\begin{definition}[Downward Markov chains]\label{down-markov}  
Let ${\mathcal N}$ be a set of natural numbers, and ${\mathcal A}$ a subset of ${\mathcal N}$.  A \emph{downward Markov chain} on ${\mathcal N}$ with absorbing states ${\mathcal A}$ is a collection of transition probabilities $P(n \searrow m)$ for $n,m \in {\mathcal N}$ obeying the following axioms:
\begin{itemize}
    \item[(i)] One has $P(n \searrow m) \geq 0$ for all $n,m \in {\mathcal N}$.  Furthermore, $P(n \searrow m)$ vanishes unless either $n \in {\mathcal N} \backslash {\mathcal A}$ and $m = n/q$ for some natural number $q > 1$, or $n \in {\mathcal A}$ and $m=n$.
    \item[(ii)]  One has
\begin{equation}\label{searrow}
\sum_m P(n \searrow m) = 1
\end{equation}
    for all $n \in {\mathcal N}$.  (In particular, this forces $P(n \searrow n) = 1$ for an absorbing state $n \in {\mathcal A}$.)  Here we adopt the convention that $P(n \searrow m)$ vanishes if $n$ or $m$ lies outside ${\mathcal N}$.
\end{itemize}
\end{definition}

\begin{example}[Randomly dividing out a prime]\label{Random} One can create a downward Markov chain on $\N$ with absorbing state $\{1\}$ by defining $P(n \searrow n/p)$ to equal $\frac{1}{\omega(n)}$ whenever $n$ is divisible by a prime $p$ (with $\omega(n)$ the number of distinct prime factors of $n$), and also defining $P(1 \searrow 1) = 1$, with $P(n \searrow m) = 0$ in all other cases.
\end{example}

\begin{example}[Mertens downward chain]\label{mertens}  The \emph{Mertens downward chain} on $\N$ with absorbing state $\{1\}$ is defined to be the downward Markov chain with $P(n \searrow m)$ defined to equal $1$ if $m$ is of the form $n/P(n)$ with $P(n)$ the largest prime factor of $n$ (with the convention $P(1)=1$), and defined to equal $0$ otherwise.  This can be easily verified to be a downward Markov chain (albeit one which is deterministic, since each state $n$ transitions to exactly one state $m$).  This chain can retroactively be viewed as implicit in much of the previous literature on primitive sets \cite{erdos35,LP, Lichtman, Lalmost,glw}.
\end{example}

The next example of a downward Markov chain is fundamental to our arguments; most of the other downward chains we consider in this paper will be slight modifications of this chain.

\begin{example}[von Mangoldt downward chain]\label{vmdc}  The \emph{von Mangoldt downward chain} on $\N$ with absorbing state $\{1\}$ is defined to be the downward Markov chain with transition probabilities
\begin{equation}\label{vmdc-def}
P( n \searrow n/q ) \coloneq \frac{\Lam(q)}{\log n}
\end{equation}
for $n \geq 2$ and $q|n$ as well as $P(1 \searrow 1) = 1$, with all other transitions $P(n \searrow m)$ vanishing.  It is easy to verify using \eqref{vmi} that this is a downward Markov chain.
\end{example}

Suppose we have a downward Markov chain on ${\mathcal N}$ with absorbing states ${\mathcal A}$.  Then every $n_0 \in {\mathcal N}$ generates a random downward divisibility chain \eqref{div-down}
where each $n_{i+1}$ is obtained from $n_i$ by a Markov process with transition probability $P(n_i \searrow n_{i+1})$.  Note that such chains will strictly decrease until they reach an absorbing state, at which point they stay constant; see Figure \ref{fig-down}.  We let $\mathbb{P}_{n_0 \searrow}$ denote the probability measure associated to such a downward chain.  If $A$ is a primitive set contained in the non-absorbing states ${\mathcal N} \backslash {\mathcal A}$, then it can only meet such a chain at most once, so on taking probabilities, we see that
\begin{equation}\label{psum}
 \sum_{k=0}^\infty \mathbb{P}_{n_0\searrow}( n_k \in A) \leq 1.
\end{equation}

\begin{figure}
  \centering
  \includegraphics[width=0.75\textwidth]{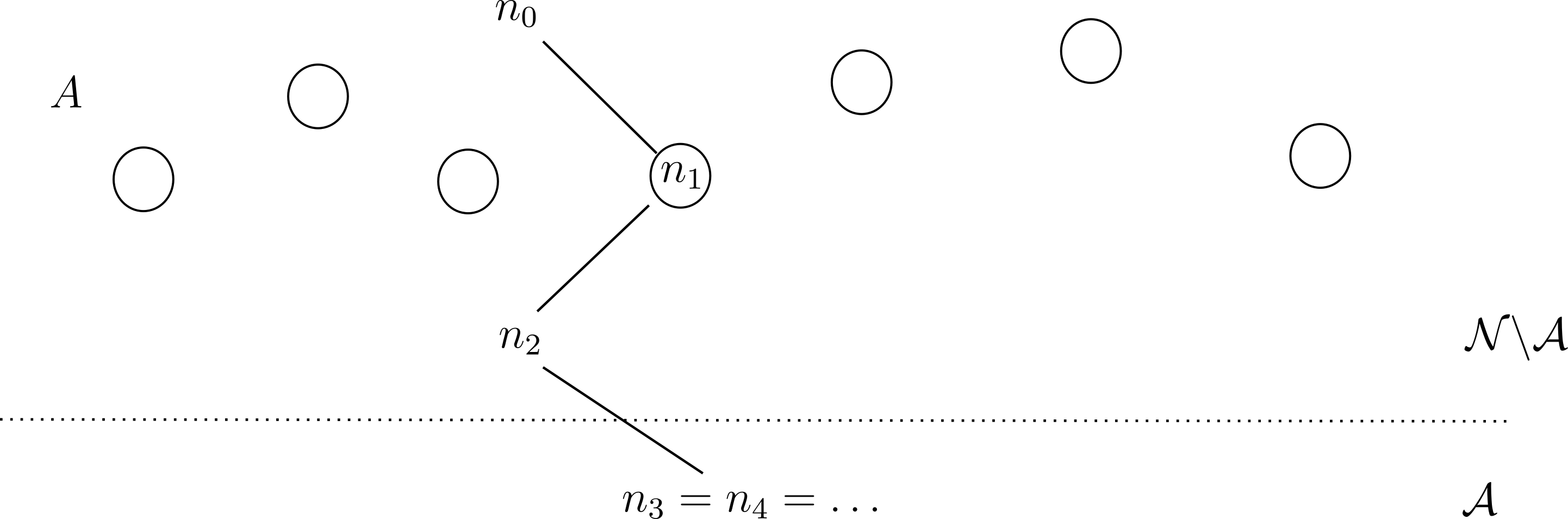}
  \caption{A downward divisibility chain that reaches an absorbing state $n_3 \in {\mathcal A}$.  Any primitive set $A$ avoiding the absorbing states ${\mathcal A}$ will meet this chain at most once, although the chain could conceivably ``jump over'' such a set.}
  \label{fig-down}
\end{figure}

\begin{example}\label{chain-ex} If $n_0=12$, then the Mertens downward chain (Example \ref{mertens}) almost surely produces the downward divisibility chain
$$ \dots | 1 | 2 | 4 | 12.$$
The von Mangoldt downward chain produces this divisibility chain with probability $\frac{\log 3}{\log 12} \frac{\log 2}{\log 4}$, but produces other chains as well; for instance, it produces the chain
\begin{equation}\label{jump}
\dots | 1 | 1 | 3 | 12
\end{equation}
with probability $\frac{\log 2}{\log 12}$.  The Mertens downward divisibility chains have the additional property of being \emph{$L$-divisible} in the sense of \cite{Lichtman}, and are thus particularly well suited for studying that concept; but we will not make use of $L$-divisibility\footnote{Indeed, it was observed in \cite[Proposition 5.3]{Lichtman} that the analogue of Theorem \ref{conj:1196} for $L$-primitive sets fails by a factor of $e^\gamma$.  In retrospect, this helps explain why it becomes advantageous to move away from the Mertens chain and consider the von Mangoldt chain instead.} in this paper.  

The von Mangoldt chain is particularly tractable for calculations, thanks to identities such as \eqref{vmi} and \eqref{lam-s}.  However, it has the defect that it can ``jump over'' primitive sets, due to the fact that $\Lambda$ is supported on prime powers in addition to primes.  For example, the chain \eqref{jump} jumps over the primitive set $\N_2$ (cf., Figure \ref{fig-divis}).  In several of our arguments we will need to modify the von Mangoldt chain slightly in order to eliminate such jumps.
\end{example}

In many cases we will not start such chains at a single point $n_0$, but instead over a range of starting points $n_0 \in {\mathcal N}$, each being assigned an initial mass $b(n_0) \geq 0$; we do not require the total mass $\sum_{n_0 \in {\mathcal N}} b(n_0)$ to equal one.  For any such assignment $b$, we see from \eqref{psum} and Tonelli's theorem that
\begin{equation}\label{psum-weight}
\sum_{n \in A} h_{b\searrow}(n) \leq \sum_{n_0 \in {\mathcal N}} b(n_0),
\end{equation}
where the \emph{downward hitting mass} $h_{b\searrow}(n)$ at a non-absorbing state $n \in {\mathcal N} \backslash {\mathcal A}$ is defined by the formula
$$ h_{b\searrow}(n) \coloneq \sum_{n_0\in\mathcal N}b(n_0)\sum_{k=0}^\infty \mathbb{P}_{n_0\searrow}( n_k = n ).$$
Informally, $h_{b\searrow}(n)$ is the total mass of $b$ that is transferred through $n$ by the downward divisibility chain.  If $b$ is normalized to have unit mass, then $h_{b\searrow}(n)$ is also the probability that a random downward chain with $n_0$ drawn using $b$ as the probability distribution contains $n$.

Note that in order for the chain to reach $n$, it either has to start at $n$ or pass through a parent $nq$ of $n$ in the chain.  This gives the recursive identity
\begin{equation}\label{h-recurse}
 h_{b\searrow}(n) = b(n) + \sum_{q \geq 2: nq \in {\mathcal N}} h_{b\searrow}(nq) P(nq \searrow n)
\end{equation}
for any $n \in {\mathcal N} \backslash {\mathcal A}$.

\begin{example}  Let $n_0 = p_1 \dots p_N$  be the product of $N$ distinct primes, and consider the downward flow of dividing out by a random prime given by Example \ref{Random}.  We let $b(n) = 1_{n=n_0}$ be the Kronecker mass at $n_0$.  A straightforward calculation then shows that
$$ h_{b \searrow}(n) = \frac{1}{\binom{N}{\omega(n)}}$$
for all $n|n_0$, leading to the \emph{LYM inequality} \cite{lubell}, \cite{meshalkin}, \cite{yamamoto}
$$ \sum_{n \in A: n | n_0} \frac{1}{\binom{N}{\omega(n)}} \leq 1$$
for any primitive set $A$.  As is well-known, this implies the \emph{Sperner inequality} \cite{sperner}
$$ \# \{ n \in A : n | n_0\} \leq \binom{N}{\lfloor N/2 \rfloor}.$$
Indeed, the arguments here can be viewed as generalizations of those in \cite{lubell}.
\end{example}

A key concept in this paper will be that of a (sub-)invariant weight.

\begin{definition}[Invariant and sub-invariant weights]  Let $P(n \searrow m)$ be a downward Markov chain on some set ${\mathcal N}$ of natural numbers with absorbing states ${\mathcal A}$, and let $\nu \colon {\mathcal N} \to (0,+\infty)$ be a weight.  We say that $\nu$ is an \emph{invariant weight} for this chain if one has the identity
\begin{equation}\label{inv}
\sum_{q > 1}  \nu(nq) P( nq \searrow n ) = \nu(n)
\end{equation}
for all $n \in {\mathcal N}$, and a \emph{sub-invariant weight} if one has the inequality
\begin{equation}\label{sub}
\sum_{q > 1} \nu(nq) P( nq \searrow n )  \leq \nu(n)
\end{equation}
for all $n \in {\mathcal N}$.
\end{definition}

\begin{example}\label{invariants}  The \emph{Mertens weight}
$$ \nu_{\mathrm{Mertens}}(n) \coloneq \frac{e^\gamma}{n} \prod_{p < P(n)} \left(1-\frac{1}{p}\right)$$
can easily be verified (by telescoping series) to be an invariant weight for the Mertens downward chain. These weights appear in the literature \cite{LP, Lichtman, Lalmost,glw}, and (excepting the $e^\gamma$ constant) originate with \Erdos{}' seminal paper \cite{erdos35}.

The \emph{von Mangoldt weight}\footnote{This closed form for the invariant weight was discovered independently by Will Sawin and the second author, while the abstract existence of such a weight was conjectured by the eighth author and established by George Lowther; see \url{https://www.erdosproblems.com/forum/thread/1196}.}
\begin{equation}\label{mang-def}
 \nuMangoldt(1)\coloneq 1,\qquad \nuMangoldt(n) \coloneq \int_1^\infty \frac{\log n}{\zeta(s) n^s}\,\dd s\quad (n>1)
 \end{equation}
is similarly an invariant weight for the von Mangoldt downward chain.  Indeed, for any $n>1$ we can use \eqref{vmdc-def}, \eqref{lam-s}, and integration by parts to compute
\begin{align*}
    \sum_{q > 1} \nuMangoldt(nq) P( nq \searrow n ) 
    &= \int_1^\infty \sum_q \frac{\Lambda(q)}{\zeta(s) (nq)^s}\,\dd s \\
    &= - \int_1^\infty \frac{\zeta'(s)}{\zeta(s)^2 n^s}\,\dd s \\
    &= \int_1^\infty \frac{\log n}{\zeta(s) n^s}\,\dd s \\
    &= \nuMangoldt(n)
\end{align*}
as desired, while for $n=1$ a similar calculation gives
$$     \sum_{q > 1} \nuMangoldt(q) P( q \searrow 1 )  = - \int_1^\infty \frac{\zeta'(s)}{\zeta(s)^2}\,\dd s = 1 = \nuMangoldt(1)$$
since $\frac{1}{\zeta(s)}$ increases from $0$ to $1$ as $s$ ranges from $1$ to infinity.  Using the standard asymptotic $\frac{1}{\zeta(s)} = s-1 - \gamma (s-1)^2 +O((s-1)^3)$ for $s>1$ and a routine calculation, one can establish the asymptotic
\begin{equation}\label{num-asym}
\nuMangoldt(n) = \left(1 - \frac{2\gamma}{\log n} + O\left(\frac{1}{\log^2 n}\right) \right) \nu_0(n)
\end{equation}
for all $n \in \N_{\geq 1}$.  In Lemma \ref{lem:mangoldt-tail-and-B}(ii) below we will show that $\nu_0$ is also a sub-invariant weight for the von Mangoldt process.
\end{example}

\subsection{Upward chains}

In several of our arguments, it will be more convenient to work with an upward chain than a downward one.

\begin{definition}[Upward Markov chains]\label{up-markov}  
Let ${\mathcal N}$ be a set of natural numbers, and introduce an absorbing state $\infty$.  An \emph{upward Markov chain} on ${\mathcal N} \cup \{\infty\}$ is a collection of transition probabilities $P(n \nearrow m)$ for $n,m \in {\mathcal N} \cup \{\infty\}$ obeying the following axioms:
\begin{itemize}
    \item[(i)] One has $P(n \nearrow m) \geq 0$ for all $n,m \in {\mathcal N} \cup \{\infty\}$.  Furthermore, $P(n \nearrow m)$ vanishes unless either $n, m \in {\mathcal N}$ and $m = nq$ for some natural number $q>1$, or $m = \infty$.
    \item[(ii)]  One has
\begin{equation}\label{nearrow}
\sum_m P(n \nearrow m) = 1
\end{equation}
    for all $n \in {\mathcal N} \cup \{\infty\}$.  (In particular, this forces $P(\infty \nearrow \infty) = 1$.) Here we adopt the convention that $P(n \nearrow m)$ vanishes if $n$ or $m$ lies outside ${\mathcal N} \cup \{\infty\}$.
\end{itemize}
\end{definition}

\begin{example}[Multiplicative simple random walk]\label{msrw}  Let $w$ be a probability measure on $\N_{\geq 1}$; thus $w(q) \geq 0$ for all $q \geq 2$ and $\sum_{q=2}^\infty w(q)=1$.  Then one can define an upward Markov chain on $\N \cup \{\infty\}$ by setting $P(n \nearrow nq) = w(q)$ for all $n \in \N$ and $q \geq 2$, $P(\infty \nearrow \infty)=1$, and all other transition probabilities equal to zero.  In this case, the absorbing state $\infty$ is never reached from any natural number.
\end{example}

The most important way to generate an upward Markov chain in our arguments will be by taking an ``adjoint'' of a downward Markov chain with respect to a sub-invariant weight.

\begin{example}[Adjoint chain]\label{adjoint}  Let $P(n \searrow m)$ be a downward Markov chain for some set ${\mathcal N}$ of natural numbers with absorbing states ${\mathcal A}$, and suppose that $\nu$ is a sub-invariant weight for this chain.  We define the \emph{adjoint} of this chain with respect to $\nu$ to be the upward Markov chain on ${\mathcal N} \cup \{\infty\}$ defined by setting
\begin{equation}\label{adjoint-def}
 P(n \nearrow m) \coloneq \frac{\nu(m)}{\nu(n)} P(m \searrow n)
\end{equation}
if $n,m \in {\mathcal N}$ with $m \neq n$, with $P(n \nearrow n) = 0$,
\begin{equation}\label{infty-transition}
P(n \nearrow \infty) \coloneq 1 - \sum_{m \neq n} \frac{\nu(m)}{\nu(n)} P(m \searrow n),
\end{equation}
$P(\infty \nearrow \infty) = 1$, and $P(\infty \nearrow m)=0$ for all $m \in {\mathcal N}$.  One easily verifies that this is indeed an upward Markov chain (with the non-negativity of $P(n \nearrow \infty)$ following from \eqref{sub}).  From \eqref{searrow} one has the identity
\begin{equation}\label{nu-recurse}
 \nu(n) = \sum_{q|n; q>1} \nu\left(\frac{n}{q} \right) P\left(\frac{n}{q} \nearrow n\right) 
\end{equation}
for all $n \in {\mathcal N} \backslash {\mathcal A}$. 
\end{example}

\begin{figure}
  \centering
  \includegraphics[width=0.75\textwidth]{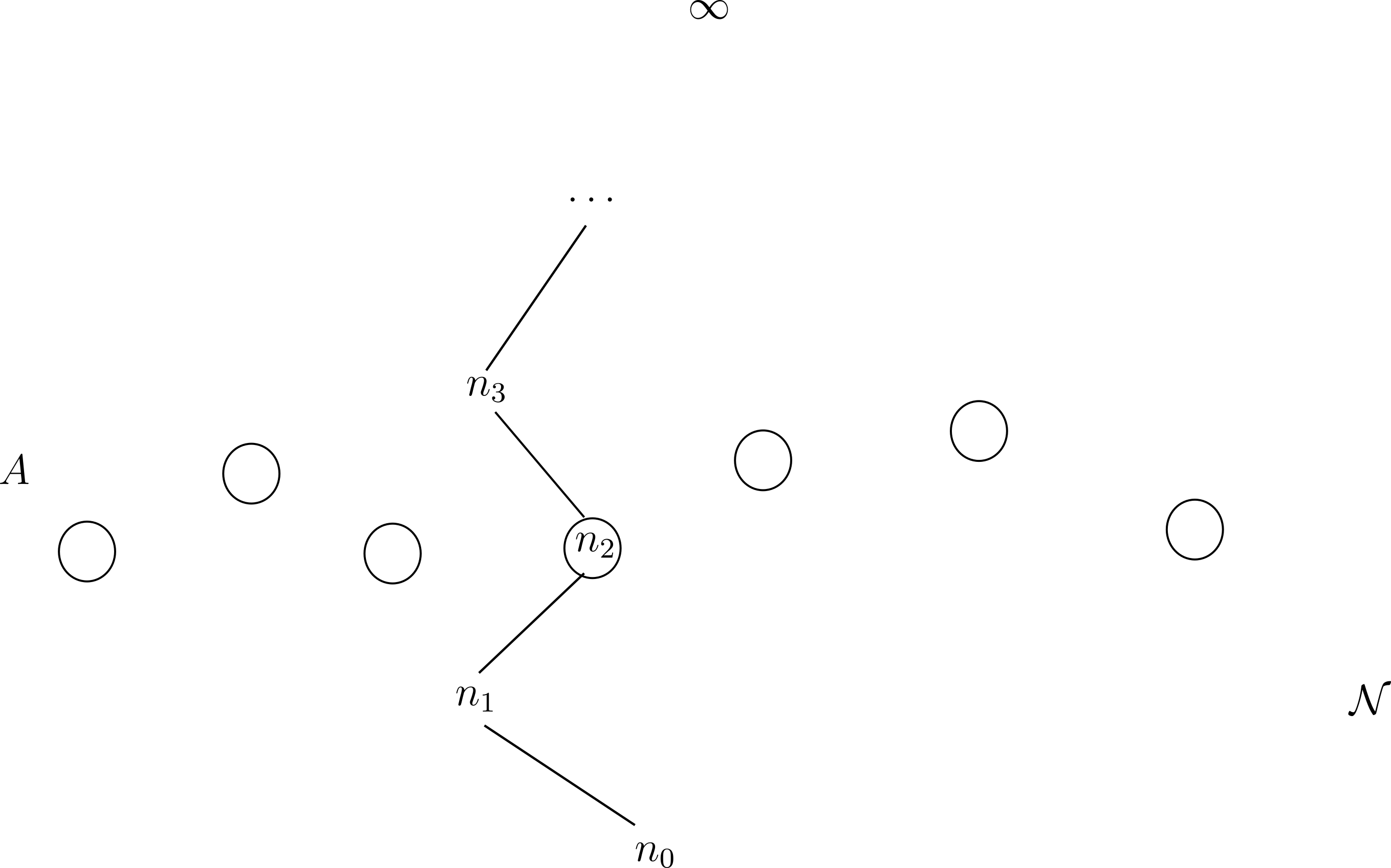}
  \caption{An upward divisibility chain, which either reaches the absorbing state $\infty$ in finite time, or is an infinite strictly increasing chain of natural numbers. Any primitive set $A$ will meet such a chain at most once, although the chain could conceivably ``jump over'' such a set.}
  \label{fig-up}
\end{figure}

Suppose we have an upward Markov chain on ${\mathcal N} \cup \{\infty\}$.  Then every $n_0 \in {\mathcal N}$ generates a random upward divisibility chain \eqref{div-up}
where each $n_{i+1}$ is obtained from $n_i$ by a Markov process with transition probability $P(n_i \nearrow n_{i+1})$, and we adopt the convention that $n | \infty$ for all $n \in \N \cup \{\infty\}$.  Note that such chains will strictly increase unless they reach $\infty$, at which point they stay constant.
We let $\mathbb{P}_{n_0\nearrow}$ denote the probability measure associated to such an upward chain.  If $A$ is a primitive set contained in ${\mathcal N}$, then it can only meet such a chain at most once, so on taking probabilities, we obtain the analogues
\begin{equation}\label{psum-upper}
 \sum_{k=0}^\infty \mathbb{P}_{n_0\nearrow}( n_k \in A) \leq 1
\end{equation}
and
\begin{equation}\label{psum-weight-upper}
\sum_{n \in A} h_{b\nearrow}(n) \leq \sum_{n_0\in\mathcal N}b(n_0)
\end{equation}
of \eqref{psum}, \eqref{psum-weight} respectively for any $n_0 \in {\mathcal N}$ and $b \colon {\mathcal N} \to [0,+\infty)$, where
the \emph{upward hitting mass} $h_{b\nearrow}(n)$ is given by the formula 
$$ h_{b\nearrow}(n) \coloneq \sum_{n_0 \in {\mathcal N}} b(n_0) \sum_{k=0}^\infty \mathbb{P}_{n_0\nearrow}( n_k = n ).$$
Analogously to \eqref{h-recurse}, we have the recursive identity
\begin{equation}\label{h-recurse-upper}
 h_{b\nearrow}(n) = b(n) + \sum_{q > 1: n/q \in {\mathcal N}} h_{b\nearrow}\left(\frac{n}{q}\right) P\left(\frac{n}{q} \nearrow n\right)
\end{equation}
for any $n \in {\mathcal N}$.

\section{Preliminary estimates}

\subsection{Bounds involving the von Mangoldt function}

We will need to estimate several sums involving the von Mangoldt function $\Lambda$.  We recall Mertens' theorems:

\begin{theorem}[Mertens' theorems]\label{Mertens}\ 
\begin{itemize}
    \item[(i)]  We have $\sum_{p \leq x} \frac{\log p}{p}, \sum_{n \leq x} \frac{\Lambda(n)}{n} = \log x + O(1)$ for all $x \geq 1$.
    \item[(ii)]  We have $\sum_{p \leq x} \frac{1}{p} = \log\log x + O(1)$ for all $x \geq 2$.
    \item[(iii)]  We have $\prod_{p \leq x} (1 - \frac{1}{p}) =  \frac{1+o(1)}{e^\gamma \log x}$ as $x \to \infty$.
\end{itemize}
\end{theorem}

\begin{proof} See for instance \cite[Theorem 2.7]{montgomeryvaughan}.
\end{proof}

Next, we record a standard bound on the Dirichlet series \eqref{lam-s}:

\begin{lemma}[Upper bound on von Mangoldt Dirichlet series]\label{lem:phi}
For every $u>0$, we have
\begin{equation}\label{phi-ineq}
\sum_{q}\frac{\Lam(q)}{q^{1+u}}
=-\frac{\zeta'(1+u)}{\zeta(1+u)}\le \frac{\log 2}{2^u-1} \leq \frac{1}{u}.
\end{equation}
(See Figure \ref{fig-phi}.)
\end{lemma}

\begin{figure}
  \centering
  \includegraphics[width=0.75\textwidth]{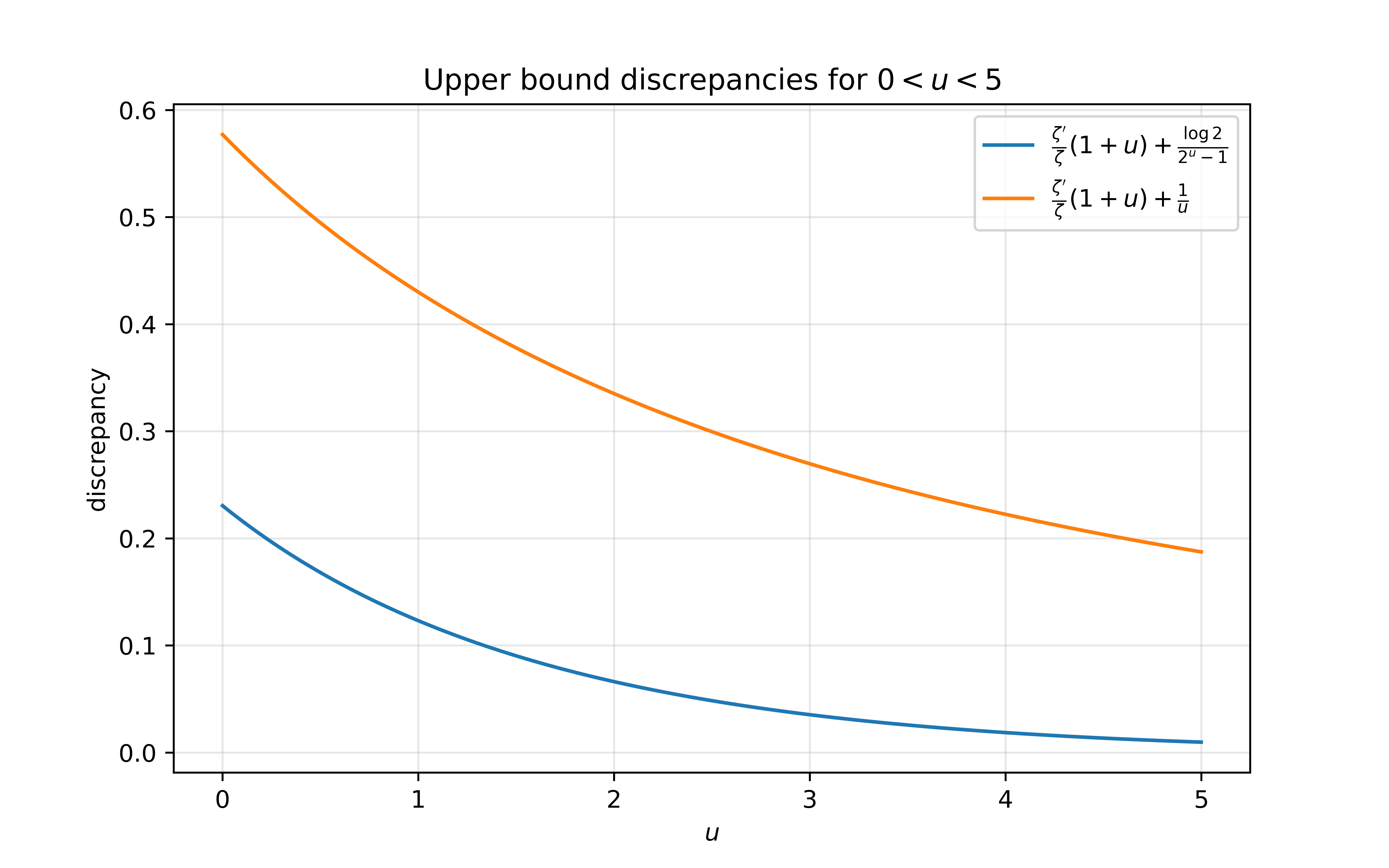}
  \caption{A plot of the discrepancy between $\frac{\log 2}{2^u-1}$ or $\frac{1}{u}$ and $-\frac{\zeta'}{\zeta}(1+u)$, for $0 < u < 5$.}
  \label{fig-phi}
\end{figure}

\begin{figure}
  \centering
  \includegraphics[width=0.75\textwidth]{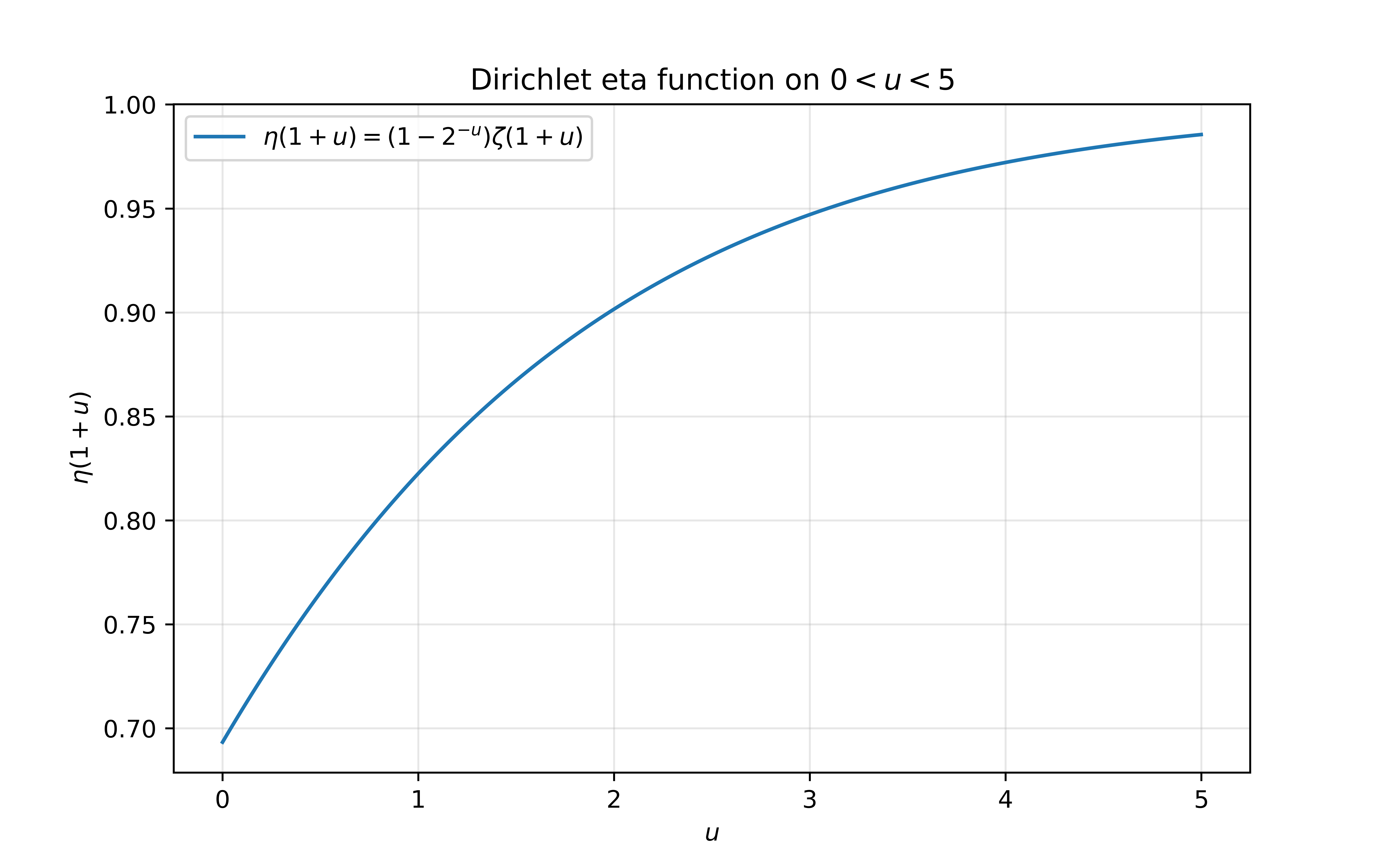}
  \caption{The Dirichlet eta function $\eta$ is strictly increasing, concave, and log-concave on the positive real axis.}
  \label{fig-eta}
\end{figure}

\begin{proof} The first identity is \eqref{lam-s}, while the final inequality follows from the calculation
$$ \frac{2^u-1}{\log 2} = u \int_0^1 2^{tu}\,\dd t \geq u$$
(one could also use the mean value theorem here if desired).
It remains to prove the middle inequality. Observe that
\begin{align*}
\frac{\zeta'(1+u)}{\zeta(1+u)}+\frac{\log 2}{2^u-1} = \frac{\eta'(1+u)}{\eta(1+u)}
\end{align*}
where $\eta$ is the Dirichlet eta function
\begin{equation}\label{etas}\eta(s)\coloneq (1-2^{1-s})\zeta(s).
\end{equation}
Thus, it will suffice to show that $\eta$ is non-decreasing for $s > 1$.  In fact, $\eta$ is known to be strictly increasing, concave, and log-concave for $s>0$ \cite{adell-lekuona}, \cite{alzer-kwong}, \cite{vdl}, \cite{wang}; see Figure \ref{fig-eta}.  For the convenience of the reader, we present here a short probabilistic proof of the non-decreasing nature of $\eta$ for $s>1$, following the arguments in \cite{adell-lekuona}.  We can write the well-known Mellin transform representation
\begin{align*}
\eta(s)= \frac{1}{\Gamma(s)} \int_0^\infty \frac{x^{s-1}}{e^x+1}\,\dd x.
\end{align*}
in probabilistic form as
\begin{align*}
\eta(s)=\E\,h(X_s)
\end{align*}
where $h(x)\coloneq 1/(1+e^{-x})$ and $X_s$ is a gamma random variable of shape $s$ and scale $1$, i.e., with density $\frac{x^{s-1}e^{-x}}{\Gamma(s)}$. If $t>s>1$, then $X_t$ has the same distribution as $X_s+Y_{t-s}$ for an independent gamma random variable $Y_{t-s}$ of shape $t-s$ and scale $1$.  As $h$ is increasing, we conclude that $\eta(t)\ge \eta(s)$, giving the claimed monotonicity. 
\end{proof}

We can control further sums of $\Lambda$ by means of the easy identities
\begin{equation}\label{log-1}
\frac{1}{\log a}=\int_0^\infty a^{-u}\,\dd u
\end{equation}
and
\begin{equation}\label{log-2}
\frac{1}{\log^2 a}=\int_0^\infty u a^{-u}\,\dd u
\end{equation}
for $a>1$.  The following lemma summarizes the relevant bounds.

\begin{lemma}\label{lem:mangoldt-tail-and-B}  Let $m \geq 1$.
\begin{itemize}
    \item[(i)] (Asymptotic estimate)  Uniformly for all $2 \leq y \leq z$, one has
\begin{equation}\label{asym-1}
 \sum_{y\le q\le z}\frac{\Lam(q)}{q\log^2(mq)} =
\frac1{\log(my)}-\frac1{\log(mz)}
+
O\!\left(\frac1{\log^2(my)}\right).
\end{equation}
In particular, on sending $z \to \infty$, one has
\begin{equation}\label{asym-2}
 \sum_{q \geq y} \frac{\Lam(q)}{q\log^2(mq)} =
\frac1{\log(my)}
+
O\!\left(\frac1{\log^2(my)}\right).
\end{equation}
    \item[(ii)] (Non-asymptotic estimate)  If $m = 2^x$ for some $x \geq 1$, then
\begin{equation}\label{sharp}
\log m \cdot \sum_q \frac{\Lam(q)}{q\log^2(mq)} \le
\sum_{j\ge 1}\frac{x}{(x+j)^2}
\le
\frac{x}{x+\tfrac12}
\le 1.
\end{equation}
See Figure \ref{fig-mangoldt}.
    \item[(iii)] (Shifted non-asymptotic estimate)  One has
\begin{equation}\label{sharp-2}
\sum_q \frac{\Lam(q)}{q\log(mq) \log(2mq)} \le \frac{1}{\log(2m)}.
\end{equation}
See Figure \ref{fig-mangoldt2}.
\end{itemize}
\end{lemma}

\begin{figure}
  \centering
  \includegraphics[width=0.75\textwidth]{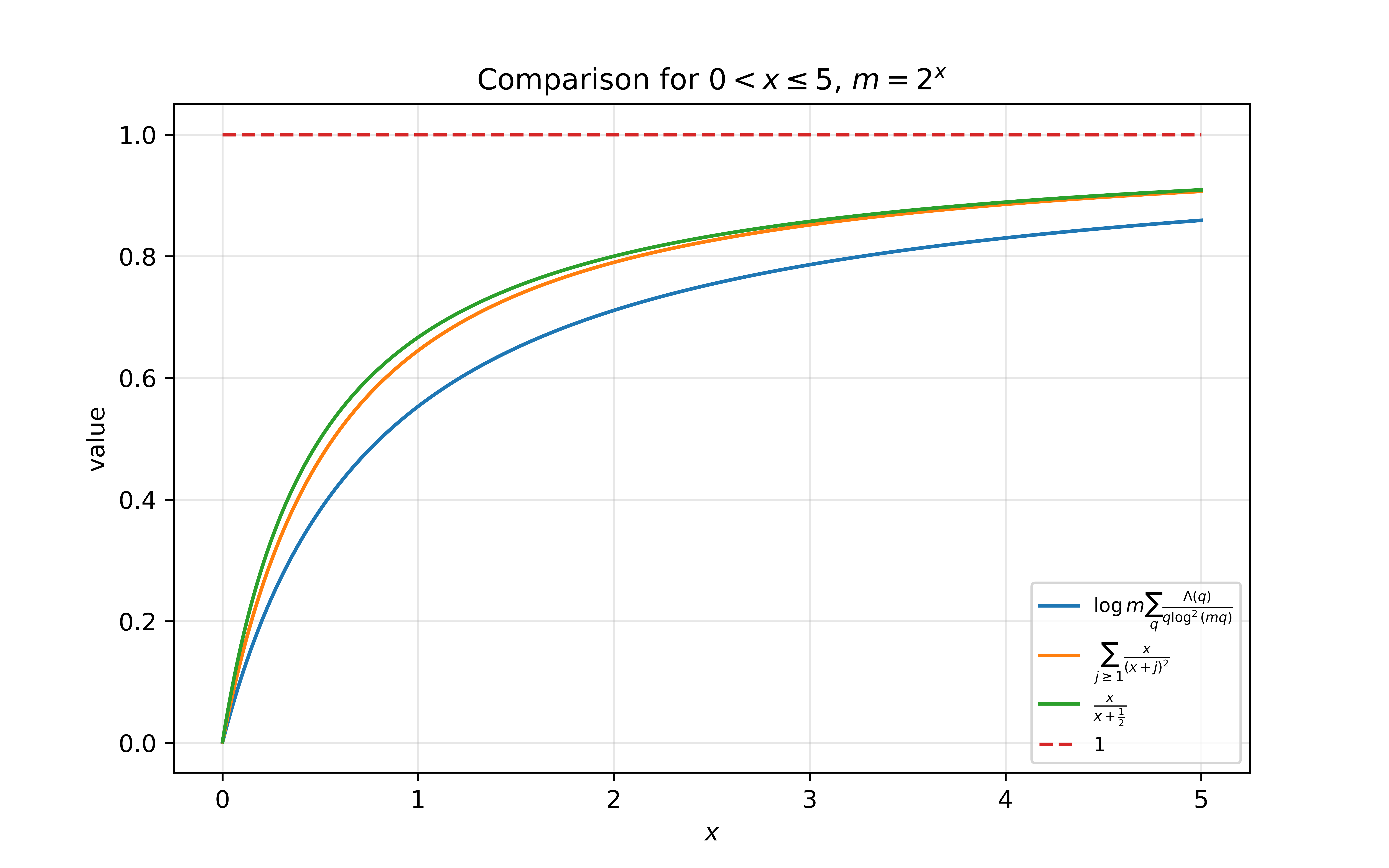}
  \caption{A comparison of the various expressions in \eqref{sharp}, for $0 < x < 5$.}
  \label{fig-mangoldt}
\end{figure}

\begin{figure}
  \centering
  \includegraphics[width=0.75\textwidth]{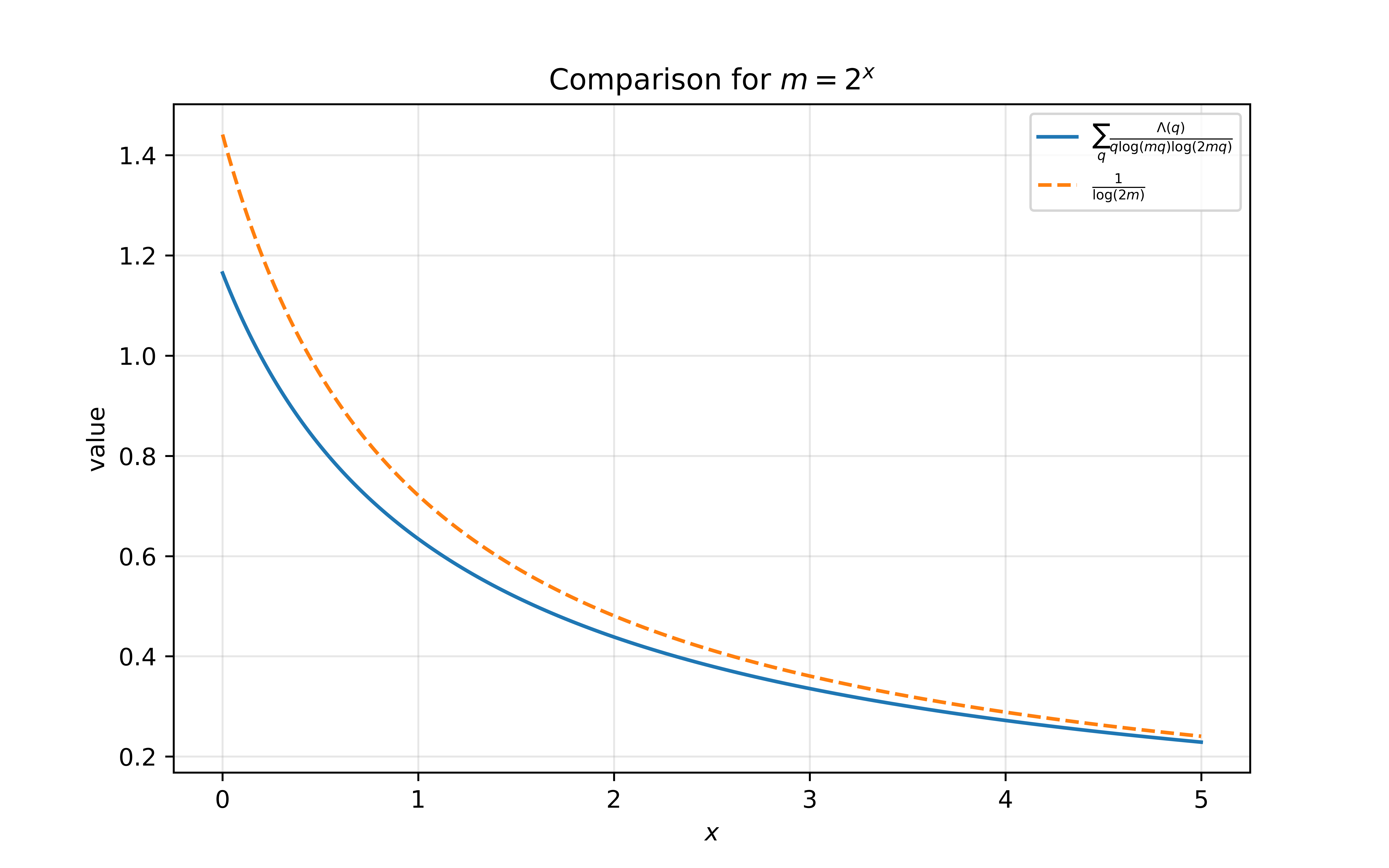}
  \caption{A comparison of the left- and right-hand sides of \eqref{sharp-2}, for $m=2^x$ and $0 < x < 5$.}
  \label{fig-mangoldt2}
\end{figure}

\begin{proof}
We first prove (i). Put $A(t)\coloneq \sum_{q\le t}\frac{\Lam(q)}{q}$ and $g(t)\coloneq \frac1{\log^2(mt)}$. By Theorem \ref{Mertens}(i), we have $A(t)=\log t+O(1)$. Partial summation followed by integration by parts then gives
\begin{align*}
 \sum_{y\le q\le z}\frac{\Lam(q)}{q\log^2(mq)}
&=
A(z)g(z)-A(y^-)g(y)-\int_y^z A(t)g'(t)\,\dd t \\
&=
\log z\,g(z)-\log y\,g(y)
-\int_y^z \log t\,g'(t)\,\dd t
+
O\!\left(g(y)+g(z)+\int_y^z |g'(t)|\,\dd t\right) \\
&= \int_y^z \frac{g(t)}{t} + O\!\left(g(y)+g(z)+\int_y^z |g'(t)|\,\dd t\right).
\end{align*}
Because $g$ is decreasing, $g(z)+\int_y^z|g'(t)|\,\dd t\ll g(y)$, and hence the error term is $O(1/\log^2(my))$. Direct calculation gives $$\int_y^z \frac{g(t)}{t}\,\dd t=\int_y^z \frac{\dd t}{t\log^2(mt)}=\frac1{\log(my)}-\frac1{\log(mz)}$$ which yields \eqref{asym-1}.  The claim \eqref{asym-2} then follows by sending $z$ to infinity.

Now we prove (ii).  Using Lemma~\ref{lem:phi}, \eqref{log-2}, the geometric series formula $\frac{1}{2^u-1} = \sum_{j \geq 1} 2^{-ju}$ and two applications of Tonelli's theorem, we have
\begin{align*}
\log m \cdot \sum_q \frac{\Lam(q)}{q\log^2(mq)} 
&\le
\log m\int_0^\infty u\,m^{-u}\sum_{q}\frac{\Lam(q)}{q^{1+u}}\,\dd u
\le
(\log 2)\log m\int_0^\infty \frac{u\,m^{-u}}{2^u-1}\,\dd u \\
&\le
(\log 2)\log m
\sum_{j\ge 1}\int_0^\infty u\,(m2^j)^{-u}\,\dd u 
=
(\log 2)\log m\sum_{j\ge 1}\frac{1}{\log^2(m2^j)} 
=
\sum_{j\ge 1}\frac{x}{(x+j)^2}.
\end{align*}
From the convexity of $t \mapsto \frac{x}{(x+t)^2}$ for $t \geq 0$ one has the inequality
$$ \frac{x}{(x+j)^2} \leq \int_{j-\frac12}^{j+\frac12} \frac{x}{(x+t)^2}\,\dd t$$ 
and hence $$\sum_{j\ge 1}\frac{x}{(x+j)^2} \leq \int_{1/2}^\infty \frac{x}{(x+t)^2}\,\dd t = \frac{x}{x+\tfrac12}.$$ This gives (ii).

Finally, we prove (iii).  From \eqref{log-1} we have
$$ \sum_q \frac{\Lam(q)}{q\log(mq) \log(2mq)} = \int_0^\infty \int_0^\infty \sum_q \frac{\Lam(q)}{q^{1+u+v} m^{u+v} 2^v}\dd{u}\dd v;$$
substituting $r=u+v$, evaluating the $v$ integral, and applying Lemma~\ref{lem:phi}, we conclude that
$$ \sum_q \frac{\Lam(q)}{q\log(mq) \log(2mq)} \leq \int_0^\infty \frac{\log 2}{2^r-1} \frac{1}{m^r} \frac{1-2^{-r}}{\log 2}\,\dd r.$$
By \eqref{log-1} the right-hand side is $\frac{1}{\log(2m)}$, giving the claim.
\end{proof}

\subsection{A further sum over primes}

The following estimate will be used to establish the odd Banks--Martin conjecture (Theorem \ref{conj:oddBM}).

\begin{lemma} \label{lem:oddzeta}
    For every $u > 0$,
\begin{equation}\label{om2}
u \sum_{p \geq 3} \frac{\log p}{(p-2) p^u} \leq 1.
\end{equation}
(See Figure \ref{fig-primesum2}.)
\end{lemma}

\begin{proof}  The functions $u \mapsto \frac{u}{p^u}$ are decreasing for $u > \frac{1}{\log 3}$ and $p \geq 3$, so we may assume without loss of generality that $0 < u \leq \frac{1}{\log 3} = 0.9102\dots$.  From \eqref{phi-ineq} and the geometric series formula we have
$$ \sum_p \frac{\log p}{p^{1+u}-1} \leq \frac{\log 2}{2^u-1}$$
and hence
\begin{align*}
\sum_{p \geq 3} \frac{\log p}{(p-2) p^u} &\leq \frac{\log 2}{2^u-1} - \frac{\log 2}{2^{1+u}-1} + \sum_{p \geq 3} \frac{\log p}{(p-2) p^u} - \frac{\log p}{p^{1+u}-1} \\
    &= \frac{2^u \log 2}{(2^u-1) (2^{1+u}-1)} + \sum_{p \geq 3} \frac{(2p^u - 1)\log p }{(p-2)p^u (p^{1+u}-1)}.
\end{align*}
The function $x \mapsto \frac{2x-1}{x(px-1)}$ has a derivative of $-\frac{2xp(x-1)+1}{x^2 (px-1)^2}$, which is negative for $p \geq 3$ and $x \geq 1$.  Thus the summand $\frac{ (2p^u - 1)\log p }{(p-2)p^u (p^{1+u}-1)}$ is decreasing in $u$, so that
$$ \frac{(2p^u - 1)\log p }{(p-2)p^u (p^{1+u}-1)} \leq \frac{\log p}{(p-1) (p-2)}.$$
Thus, the desired bound will follow if we have
$$
\sum_{3 \leq p \leq 7} \frac{\log p (2p^u - 1)}{(p-2)p^u (p^{1+u}-1)} + C \leq \frac{1}{u} - \frac{2^u \log 2}{(2^u-1) (2^{1+u}-1)}
$$
for $0 < u < \frac{1}{\log 3}$, where
$$ C \coloneq \sum_{p > 7} \frac{\log p}{(p-1) (p-2)} = 0.11110\dots.$$
We can use the convexity of $2^u$ to bound
\begin{equation}\label{spot}
\frac{2^u - 1}{u \log 2} = \int_0^1 2^{tu}\ dt \geq 2^{u/2}
\end{equation}
and hence
$$ \frac{1}{u} - \frac{2^u \log 2}{(2^u-1) (2^{1+u}-1)} \geq \frac{1}{u} - \frac{2^{u/2}}{u(2^{1+u}-1)}
= \frac{(2^{u/2}-1)2^{u/2} + 2^u -1}{u(2^{1+u}-1)}.$$
Applying  \eqref{spot} two further times, we conclude that
$$ \frac{1}{u} - \frac{2^u \log 2}{(2^u-1) (2^{1+u}-1)} \geq \frac{\frac{1}{2} 2^{3u/4} + 2^{u/2}}{2^{1+u}-1} \log 2$$
and so we reduce to showing that
\begin{equation}\label{analytic}
\sum_{3 \leq p \leq 7} \frac{\log p (2p^u - 1)}{(p-2)p^u (p^{1+u}-1)} + C \leq 
\frac{\frac{1}{2} 2^{3u/4} + 2^{u/2}}{2^{1+u}-1} \log 2
\end{equation}
for $0 < u < \frac{1}{\log 3}$.  This can be verified routinely by interval arithmetic; see Figure \ref{fig-primesum3}.
\end{proof}


\begin{figure}
  \centering
  \includegraphics[width=0.75\textwidth]{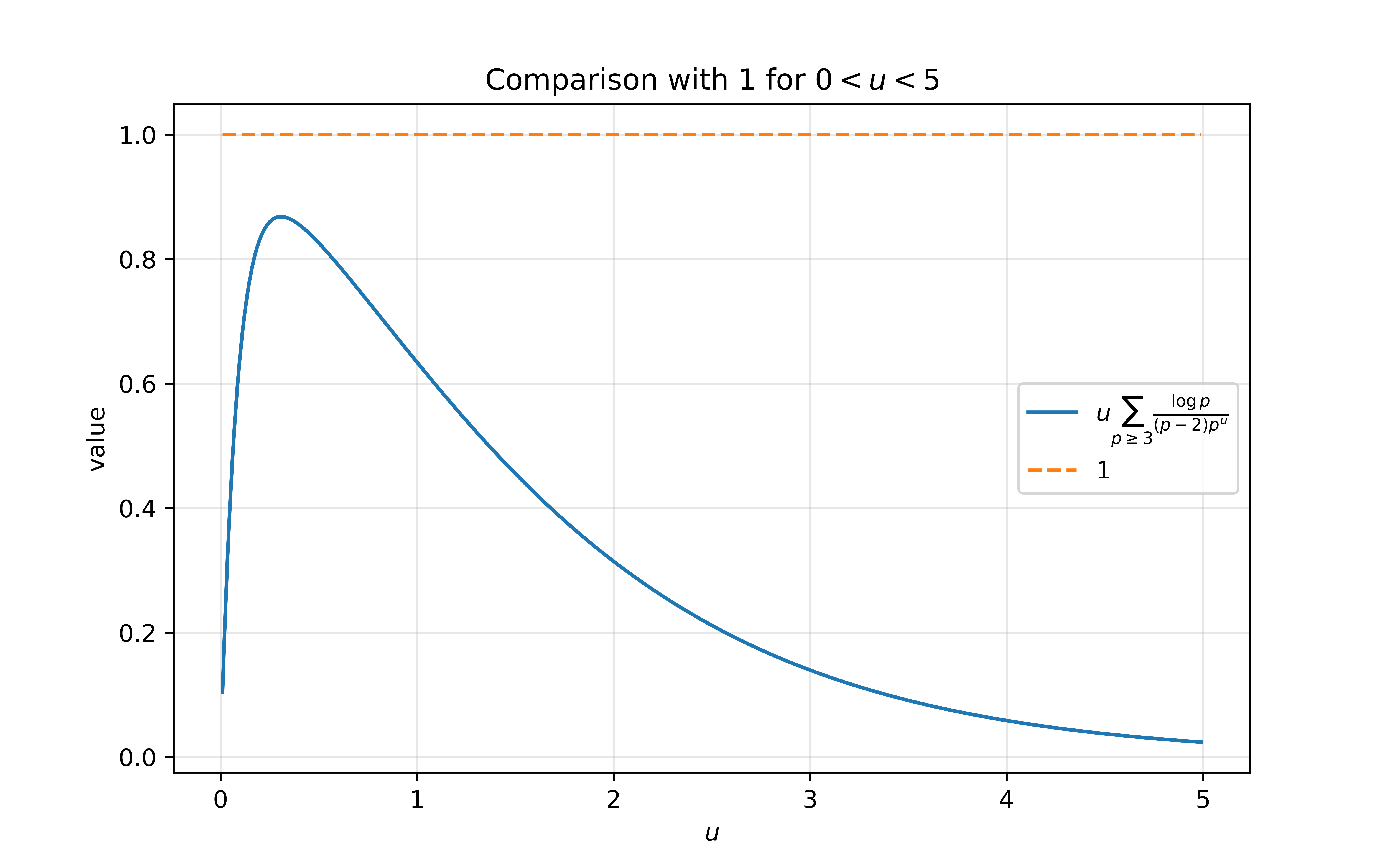}
  \caption{A comparison of the left and right-hand sides of \eqref{om2}.}
  \label{fig-primesum2}
\end{figure}

\begin{figure}
  \centering
  \includegraphics[width=0.75\textwidth]{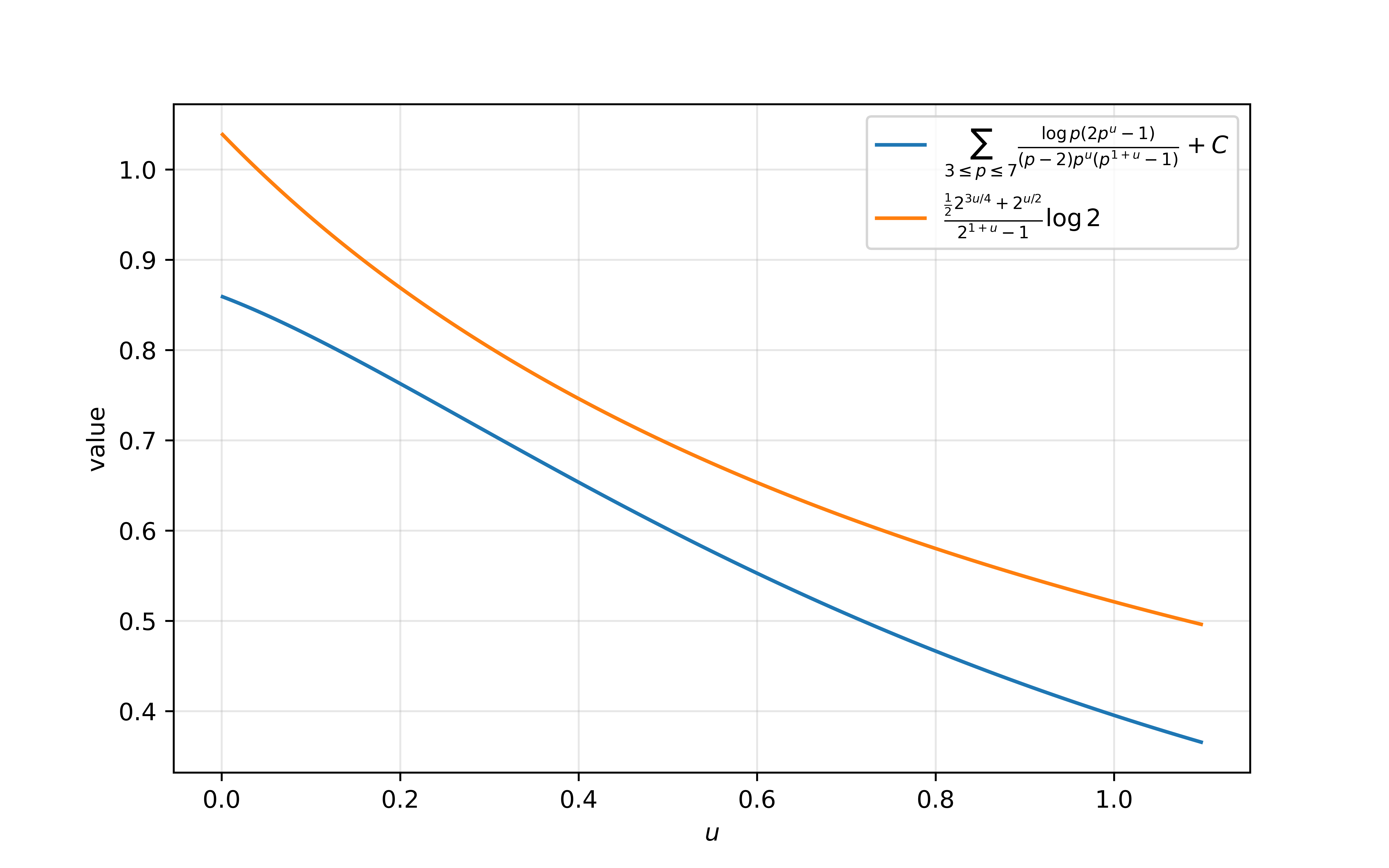}
  \caption{A comparison of the left and right-hand sides of \eqref{analytic}.}
  \label{fig-primesum3}
\end{figure}

\section{Primitive sets of large numbers (\Erdos{} \#1196)}\label{1196-sec} 

In this section, we establish Theorem~\ref{conj:1196}.  By a limiting argument, we may assume without loss of generality that $A \subset [x,X]$ for some finite $X$.

We use the von Mangoldt downward chain on $\N$ with absorbing state $\{1\}$ (Example \ref{vmdc}).  We define the initial mass $b \colon \N \to [0,+\infty)$ by setting
\begin{align}
b(n) \coloneq \nu_0(n) - \sum_{2\le q\le X/n} \nu_0(nq) P(nq \searrow n) 
\label{eq:bX-def}
\end{align}
for $x \leq n \leq X$, and $b(n)=0$ otherwise.  Observe from Lemma~\ref{lem:mangoldt-tail-and-B}(ii), \eqref{vmdc-def}, and \eqref{nu-def} that $b$ is always non-negative.  

Since $b$ is supported on $[x,X]$, all downward chains originating from an $n$ in the support of $b$ are contained in $[1,X]$.  In particular, $h_{b\searrow}(n)=0$ for $n>X$.  A routine downward induction using \eqref{eq:bX-def} and \eqref{h-recurse} then shows that $h_{b\searrow}(n) = \nu_0(n)$ for all $x \leq n \leq X$.  Applying \eqref{psum-weight} and \eqref{eq:bX-def}, we conclude that
$$ f(A) \leq \sum_{x \leq n \leq X} \nu_0(n) - \sum_{x \leq n \leq X} \sum_{2 \leq q \leq X/n} \nu_0(nq) P(nq \searrow n).$$
From \eqref{vmi}, \eqref{nu-def} (and relabeling $n$ as $r$)  we have
$$ \sum_{x \leq n \leq X} \nu_0(n) = \sum_{x \leq r \leq X} \frac{1}{r \log^2 r} \sum_{q|r} \Lambda(q)$$
while from \eqref{vmdc-def} (and relabeling $nq$ as $r$) one has
$$\sum_{x \leq n \leq X} \sum_{2 \leq q \leq X/n} \nu_0(nq) P(nq \searrow n)
= \sum_{x \leq r \leq X} \frac{1}{r \log^2 r} \sum_{\substack{q\mid r\\ r/q\ge x}}\Lam(q).$$
We conclude that
\begin{equation}\label{fA-bound}
f(A) \leq \sum_{x\le r\le X}\frac1{r\log^2 r}\sum_{\substack{q\mid r\\ r/q<x}}\Lam(q).
\end{equation}
Discarding the upper bound $r \leq X$ and making the change of variables $n=r/q$, we obtain
\begin{align*}
f(A) \leq \sum_{n<x}\frac1n \sum_{q \geq \max(2,x/n)}\frac{\Lam(q)}{q\log^2(nq)}.
\end{align*}
Since $n \max(2,x/n) \geq x$, Lemma~\ref{lem:mangoldt-tail-and-B}(i) gives
\begin{align*}
\sum_{q \geq \max(2,x/n)}\frac{\Lam(q)}{q\log^2(nq)}
&\le \frac1{\log x} + O\!\left(\frac1{\log^2 x}\right) 
\end{align*}
and thus
\begin{equation}\label{filip}
\sum_{x\le r\le X}\frac1{r\log^2 r}\sum_{\substack{q\mid r\\ r/q<x}}\Lam(q) \le \sum_{n<x}\frac1n\left(\frac1{\log x} + O\!\left(\frac1{\log^2 x}\right)\right) = 1 + O\!\left(\frac1{\log x}\right)
\end{equation}
and the claim then follows from \eqref{fA-bound}.

\begin{remark} Nat Sothanaphan, using\footnote{\url{https://drive.google.com/file/d/1Hb_0TDRTSwJpwifYVcSgkGtQEJdL5bOV/view}} GPT, has analyzed a variant of this argument to obtain the sharper bound
$$ f(A) \leq 1 + \frac{\gamma}{\log x} + O\!\left(\frac1{\log^2 x}\right).$$
In view of \eqref{nk} it is natural to conjecture that the error terms can be improved further, possibly even to one which is polynomial in $k$, although this may require the assumption of additional conjectures such as the Riemann hypothesis to prove. The von Mangoldt chain may not be the ideal tool for achieving such an improvement, due to its ability to jump over primitive sets as discussed in Example \ref{chain-ex}.
\end{remark}

\begin{remark}\label{1196-alt-rem} An alternate Markov chain proof of Theorem~\ref{conj:1196} (but with the constant $\gamma$ weakened to $2\gamma$) can be given as follows.  By taking the adjoint of the downward von Mangoldt chain from Example \ref{vmdc} with respect to the von Mangoldt weight $\nuMangoldt$ from Example \ref{invariants} one obtains an ``upward von Mangoldt chain'' $P(n \nearrow m)$ on $\N \cup \{\infty\}$ (actually one can delete the absorbing state $\infty$ if desired as it is almost surely never reached from a finite state).  Let $b \colon \N \cup \{\infty\} \to [0,+\infty)$ be the weight $b(n) \coloneqq 1_{n=1}$, then an easy induction using \eqref{nu-recurse}, \eqref{h-recurse-upper}, and $\nuMangoldt(1)=1$ gives that $h_{b\nearrow}(n) = \nuMangoldt(n)$ for all $n \in \N$.  Applying \eqref{psum-weight-upper} one obtains the inequality
$$ \sum_{n \in A} \nuMangoldt(n) \leq 1$$
for all primitive $A \subset \N$; specializing to the case $A \subset [x,\infty)$ and using \eqref{num-asym} one obtains the upper bound
$$ f(A) \leq 1 + \frac{2\gamma}{\log x} + O\!\left(\frac1{\log^2 x}\right),$$
which yields Theorem~\ref{conj:1196}.  We remark that if we replace the von Mangoldt chains and weights with their Mertens counterparts and apply Theorem~\ref{Mertens}(iii), we recover the weaker bound
\begin{equation}\label{latter}
f(A) \leq e^\gamma + o(1)
\end{equation}
which was implicitly obtained in the original paper of \Erdos{} \cite{erdos35}.  Indeed, this argument can be viewed as a re-interpretation of the arguments in \cite{erdos35}.  We remark that the above argument recovers the known fact that the latter bound \eqref{latter} extends to sets $A$ that are merely $L$-primitive instead of primitive, in which case the inequality becomes sharp up to the $o(1)$ error; see \cite[Proposition 5.3]{Lichtman}.
\end{remark}

\begin{remark}\label{GPT-rem-2}  The proof of Theorem~\ref{conj:1196} given in this section also yields the more general inequality
$$ f(A \cap \mathcal{D} \cap [x,\infty)) \leq \left(1 + O\left(\frac{1}{\log x}\right)\right) \frac{\sum_{n \in \mathcal{D} \cap [1,x]} \frac{1}{n}}{\log x}$$
for any primitive set $A \subset \N_{\geq 1}$, any $x \geq 2$, and any subset ${\mathcal D}$ of $\N$ which is \emph{divisor closed} in the sense that $n \in {\mathcal D}$ and $m|n$, then $m \in {\mathcal D}$ as well.  For instance, one could take ${\mathcal D} = \N({\mathcal Q})$ for some set of primes ${\mathcal Q}$.  As a further special case, if all the elements of the primitive set $A$ are coprime to a fixed modulus $q$ and are greater than or equal to $x$, then
$$ f(A) \leq \frac{\phi(q)}{q} + O_q\left(\frac{1}{\log x}\right).$$
Also, from the obvious inequality
\begin{equation}\label{nu-dilate}
\nu_0(dn) \leq \frac{1}{d} \nu_0(n)
\end{equation}
for any $n \in \N_{\geq 1}$ and $d \geq 1$, we have the bound
$$ f(A) \leq \frac{1}{d} + O_d\left(\frac{1}{\log x}\right)$$
when $A$ is a primitive set consisting of multiples of $d$ that are greater than or equal to $x$.
\end{remark}

\section{\Erdos{} Primitive Set Conjecture}\label{EPS-sec}

In this section, we verify Theorem \ref{conj:EPS} by the Markov chain method.

\subsection{Construction of Markov chains}

We first construct a downward Markov chain $P(n \searrow m)$ on the set $\N_{\geq 1}$, with the primes $\N_1$ as the absorbing states.  The von Mangoldt downward chain
in Example \ref{vmdc} is almost suitable for this purpose, but has the drawback that prime powers $p^k$, $k \geq 2$, can transition to $1$ with non-zero probability.  To address this issue, we perform the following modification of the von Mangoldt downward chain.  If $n \in \N_{\geq 1}$ is composite but not a prime power, we make no changes, defining
$$ P(n \searrow n/q) \coloneq \frac{\Lam(q)}{\log n}$$
whenever $q$ is a (necessarily proper) divisor of $n$, and $P(n \searrow m)=0$ for all other $m$.  If $n \in \N_{\geq 1}$ is a prime $p$, we define $P(p \searrow p) \coloneq 1$ and $P(p \searrow m) \coloneq 0$ for all other $m$. Finally, if $n = p^k$ for a prime $p$ and some $k \geq 2$, we set
$$ P(n \searrow n/q) = \frac{\Lam(q)}{\log n} = \frac{1}{k}$$
for $q = p, p^2, \dots, p^{k-2}$, and
$$ P(n \searrow n/q) = \frac{\Lam(q)}{\log n} + \frac{1}{k} = \frac{2}{k}$$
if $q = p^{k-1}$.  One can easily check from \eqref{vmi} and the identity $\tfrac{k-2}{k}+\tfrac{2}{k}=1$ that this produces a downward Markov chain on $\N_{\geq 1}$ with absorbing states $\N_1$.

We have the following crucial sub-invariance property:

\begin{lemma}[Sub-invariance]\label{rse}  The weight $\nu_0$ is sub-invariant for this downward Markov chain.
\end{lemma}

\begin{proof}  By \eqref{nu-def}, \eqref{sub} our task is to show that
\begin{equation}\label{qo}
 \sum_{q>1} \frac{\log m}{q \log(mq)} P(mq \searrow m) \leq 1
 \end{equation}
for all $m$. Observe that 
$$ \frac{\log m}{q \log(mq)} P(mq \searrow m) = \log m \cdot \frac{\Lam(q)}{q\log^2(mq)}$$
except when $m$ is a prime $p$ and $q = p^{k-1}$ is a power of that prime, in which case
$$ \frac{\log m}{q \log(mq)} P(mq \searrow m)  = \log m \cdot \frac{\Lam(q)}{q\log^2(mq)} + \frac{1}{k^2 p^{k-1}}.$$
Thus, the left-hand side of \eqref{qo} is equal to
$$ \log m\sum_{q\ge 2} \frac{\Lam(q)}{q\log^2(mq)}$$
when $m$ is composite and
$$ \log m\sum_{q\ge 2} \frac{\Lam(q)}{q\log^2(mq)} + \sum_{k=2}^\infty \frac{1}{k^2 p^{k-1}}$$
when $m$ is a prime $p$.  Applying Lemma \ref{lem:mangoldt-tail-and-B}(ii), we obtain an upper bound of $1$ when $m$ is composite and
$$ \frac{x}{x+\frac{1}{2}} + \sum_{k=2}^\infty \frac{1}{k^2 p^{k-1}}$$
when $m=p$ is prime, where $x = \frac{\log p}{\log 2}$.  Since $x \leq p-1$, we can bound
$$ \sum_{k=2}^\infty \frac{1}{k^2 p^{k-1}} \leq \frac{1}{4} \sum_{k=2}^\infty \frac{1}{p^{k-1}} = \frac{1}{4(p-1)} \leq \frac{1}{2p-1} \leq \frac{1}{2x+1} = 1 - \frac{x}{x+\frac{1}{2}},$$
so we obtain the claim in either case.
\end{proof}

Using Example \ref{adjoint}, we can now construct an adjoint upward Markov chain $P(n \nearrow m)$ on $\N_{\geq 1} \cup \{\infty\}$ to the downward Markov chain $P(n \searrow m)$ using the sub-invariant weight $\nu_0$.

\subsection{Conclusion of the argument}

Define the initial mass distribution $b \colon \N_{\geq 1} \to [0,+\infty)$ by setting $b(p) = \nu_0(p)$ when $p$ is prime and $b(n)=0$ when $n$ is composite.  From \eqref{psum-weight-upper} applied to the upward Markov chain just constructed, we conclude that
\begin{equation}\label{na}
 \sum_{n \in A} h_{b\nearrow}(n) \leq \sum_p \nu_0(p).
\end{equation}
From \eqref{h-recurse-upper} we have $h_{b\nearrow}(p) = b(p) = \nu_0(p)$ for all primes $p$.  From \eqref{nu-recurse} and induction we then see that in fact $h_{b\nearrow}(n) = \nu_0(n)$ for all $n \in \N_{\geq 1}$.  The inequality \eqref{na} then yields Theorem \ref{conj:EPS}.

\begin{remark}
The Lean formalization mentioned in Remark~\ref{remark:lean-erdos164} includes a version of the argument presented in this section.
\end{remark}

\begin{remark}\label{1196-gen} The same argument gives the more general inequality $f(A({\mathcal Q})) \leq f({\mathcal Q})$ for any primitive set $A \subset \N_{\geq 1}$ and any set ${\mathcal Q}$ of primes. This can also be deduced from Theorem~\ref{conj:1196} using the argument used to prove \cite[Theorem 2]{ezhang}.
\end{remark}

\section{Odd Banks--Martin}\label{obm-sec}

We now prove Theorem \ref{conj:oddBM}.  Let $k \geq 1$, let $\mathcal Q$ be a set of odd primes, and let $A$ be a primitive set contained in $\N_{\geq k} = \N_{>k} \cup \N_k$.
Our goal is to show that
$$ f(A(\mathcal Q)) \ \le\ f\big(\N_k(\mathcal Q)\big).$$
By replacing $A$ with $A(\mathcal Q)$ if necessary we may assume that
$$ A = A(\mathcal Q) \subset \N_{\ge k}({\mathcal Q}).$$

\subsection{Construction of Markov chains}

We construct a modification of the von Mangoldt downward chain from Example \ref{vmdc} on $\N_{\ge k}({\mathcal Q})$, with absorbing states $\N_k({\mathcal Q})$.  We cannot quite use this chain directly, due to its ability to jump over the states in $\N_k({\mathcal Q})$ as discussed in Example \ref{chain-ex}.  To get around this, we shall only permit transitions that divide out by a prime.  More precisely, we define
\begin{equation}\label{pdown}
P(n \searrow n/p) \coloneq \frac{v_p(n) \beta_p \log p}{\lambda(n)}
\end{equation}
whenever $n\in \N_{>k}({\mathcal Q})$ and $p|n$ is prime, where $v_p(n)$ is the number of times $p$ divides $n$, $\beta_p$ is the weight
$$ \beta_p \coloneq \frac{p}{p-2},$$
and $\lambda$ is the normalization factor
$$ \lambda(n) \coloneq \sum_{p|n} v_p(n) \beta_p \log p.$$
We set $P(n \searrow m)=0$ for all other $m$.  Finally, for $n \in \N_k({\mathcal Q})$, we set $P(n \searrow n)=1$ and $P(n \searrow m)=0$ for all other $m$.  It is clear that this is a downward Markov chain on $\N_{\ge k}({\mathcal Q})$, with absorbing states $\N_k({\mathcal Q})$.

We have the following analogue of Lemma \ref{rse}:

\begin{lemma}[Sub-invariance]  The weight $\nu_0$ is sub-invariant for this downward Markov chain.
\end{lemma}

\begin{proof}
We fix $m \in \N_{\geq k}({\mathcal Q})$ and abbreviate 
$$L\coloneq \log m,\qquad \lambda\coloneq \lambda(m).$$ 
From \eqref{sub}, \eqref{nu-def} and \eqref{pdown}, our task is now to show that
\begin{equation}\label{task}
 \sum_{p\in\mathcal Q}
\frac{L v_p(m) \beta_p\log p}{p(L+\log p)(\lambda+\beta_p\log p)} + \sum_{p \in {\mathcal Q}} \frac{L \beta_p\log p}{p(L+\log p)(\lambda+\beta_p\log p)} \leq 1.
 \end{equation}
To control the first term, we use the inequalities
$$L+\log p\ge L; \quad \lambda+\beta_p\log p\ge\lambda$$
and the identity
$$\frac{\beta_p}{p}=\frac1{p-2}=\frac{\beta_p-1}{2}$$
to bound
\begin{align*}
 \sum_{p\in\mathcal Q}
\frac{L v_p(m) \beta_p\log p}{p(L+\log p)(\lambda+\beta_p\log p)} &\le\frac1{\lambda}\sum_{p\mid m}\frac{v_p(m)\beta_p\log p}{p}\\
&=\frac1{2\lambda}\sum_{p\mid m}v_p(m)(\beta_p-1)\log p\\
&=\frac{1}{2\lambda} (\lambda-L) \\
&= \frac{1}{2} - \frac{L}{2\lambda}.
\end{align*}

For the second term, we instead use the identity
\begin{equation}\label{recip}
\frac1{a}=\int_0^\infty e^{-at}\,\dd t
\end{equation}
for $a = L + \log p, \lambda+\beta_p\log p$
and
$$ \beta_p = \frac{p}{p-2} > 1$$
and Tonelli's theorem to write
\begin{align*}
\sum_{p \in {\mathcal Q}} \frac{L \beta_p\log p}{p(L+\log p)(\lambda+\beta_p\log p)}  &=L\int_0^\infty\int_0^\infty
e^{-Lt-\lambda s}\sum_{p\in\mathcal Q}\frac{\beta_p\log p}{p^{1+t+\beta_p s}}\,\dd t\,\dd s \\
&= L\int_0^\infty\int_0^\infty e^{-Lt - \lambda s}\sum_{p\in\mathcal Q}\frac{\log p}{(p-2)p^{t+\beta_p s}}\,\dd t\,\dd s \\
&\leq L\int_0^\infty\int_0^\infty e^{-Lt - \lambda s}\sum_{p\ge 3}\frac{\log p}{(p-2)p^{t+s}}\,\dd t\,\dd s \\
\end{align*}
and hence by Lemma \ref{lem:oddzeta} 
\begin{align*}
\sum_p \frac{L \beta_p\log p}{p(L+\log p)(\lambda+\beta_p\log p)} &\le
L\int_0^\infty\int_0^\infty
e^{-Lt - \lambda s}\ \frac{\dd t\, \dd s}{t+s} \\
&= L\int_0^1\int_0^\infty e^{-r(L+(\lambda-L)\theta)}\,\dd r\,\dd \theta \\
&=L\int_0^1\frac{\dd\theta}{L+(\lambda-L)\theta}\\
&\leq \frac{L}{2} \left( \frac{1}{L} + \frac{1}{\lambda} \right)\\
&= \frac{1}{2} + \frac{L}{2\lambda}
\end{align*}
where we have used the change of variables $(r,\theta) = (t+s, \frac{s}{t+s})$ (so $\dd r \dd \theta = \frac{\dd t \dd s}{t+s}$), followed by \eqref{recip} and the convexity of $\theta \mapsto \frac{1}{L+(\lambda-L)\theta}$.  Summing the bounds, we obtain \eqref{task} as desired.
\end{proof}

Using Example \ref{adjoint}, we can now construct an adjoint upward Markov chain $P(n \nearrow m)$ on $\N_{\geq k}({\mathcal Q}) \cup \{\infty\}$ to the downward Markov chain $P(n \searrow m)$ using the sub-invariant weight $\nu_0$.

\subsection{Conclusion of the argument}

Define the initial mass distribution $b \colon \N_{\ge k}({\mathcal Q}) \to [0,+\infty)$ by setting $b(n) = \nu_0(n)$ when $n \in \N_k({\mathcal Q})$ and $b(n)=0$ otherwise.  From \eqref{psum-weight-upper} applied to the upward Markov chain just constructed, we conclude that
\begin{equation}\label{na-2}
 \sum_{n \in A} h_{b\nearrow}(n) \leq f( \N_{k}({\mathcal Q}) ).
\end{equation}
From \eqref{h-recurse-upper} we have $h_{b\nearrow}(n) = b(n) = \nu_0(n)$ for $n \in \N_k({\mathcal Q})$.  From \eqref{nu-recurse} and induction we then see that in fact $h_{b\nearrow}(n) = \nu_0(n)$ for all $n \in \N_{\geq k}({\mathcal Q})$.  The inequality \eqref{na-2} then yields Theorem \ref{conj:oddBM}.

\begin{remark}\label{remark:qualitative-1196}
Theorem~\ref{conj:oddBM} also implies the qualitative form of Theorem~\ref{conj:1196}, namely that every primitive set $A \subset [x,\infty)$ satisfies $f(A) \leq 1+o(1)$ as $x \to \infty$, as originally conjectured by \Erdos{}, \Sarkozy{}, and \Szemeredi{}.  This provides an alternate proof of that conjecture in its original form, albeit without the sharp $O(1/\log x)$ rate. The argument is as follows. Let $\eps>0$, let $k$ be sufficiently large depending on $\eps$, let $x$ be sufficiently large depending on $k,\eps$, and let $A \subset [x,\infty)$ be primitive. It will suffice to show that $f(A) \leq 1 + O(\eps)$.

We split $\N = \N_{<k}^o \cup \N_{\geq k}^o$, where $\N_{< k}^o$ (resp.~$\N_{\geq k}^o$) is the set of natural numbers with fewer than $k$ (resp.~at least $k$) odd prime factors, counting multiplicity.  From Theorem \ref{Mertens} we have
$$ f(\N_{<k}^o \cap [e^{e^{m-1}},e^{e^m}]) \ll \frac{1}{e^m} \sum_{n \in \N_{<k}^o \cap [1,e^{e^m}]} \frac{1}{n} \leq \frac{1}{e^m} \left( \sum_{j=0}^\infty \frac{1}{2^j} \right) \sum_{r<k} \left( \sum_{p \leq e^{e^m}} \frac{1}{p} \right)^r \ll_k \frac{m^{k-1}}{e^m}$$
for any $m \geq 1$; summing over $m \geq \log\log x$ we conclude that $f(\N_{<k}^o \cap [x,+\infty)) \leq \eps$ for $x$ large enough.  Thus, by removing the portion of $A$ that lies in $\N_{<k}^o$ we may assume that $A \subset \N_{\geq k}^o$.  This allows us to decompose $A = \bigcup_{j=0}^\infty 2^j \cdot A_j$ for some $A_j \subset \N_{\geq k}(\mathcal{Q})$, where $\mathcal{Q} \coloneq \N_1 \setminus \{2\}$ are the odd primes.  Each $A_j$ is primitive, thus by Theorem~\ref{conj:oddBM} and \eqref{nu-dilate} we conclude that
$$ f(A) \leq \sum_{j=0}^\infty \frac{f(A_j)}{2^j} \leq \sum_{j=0}^\infty \frac{f(\N_k(\mathcal{Q}))}{2^j} = 2 f(\N_k(\mathcal{Q})).$$
By \cite[Corollary~4.2]{Lalmost}, $f(\N_k(\mathcal{Q}))$ converges to $1/2$ as $k \to \infty$.  Setting $k$ sufficiently large depending on $\eps$, we obtain the claim.
\end{remark}

\begin{remark}\label{GPT-rem-3}  As mentioned in the introduction, the full Banks--Martin conjecture 
$f(A \cap \N_{\geq k}({\mathcal Q})) \leq f(\N_k({\mathcal Q}))$ can fail if ${\mathcal Q}$ is allowed to contain $2$.  However,
for $k=1$ this bound holds thanks to Remark \ref{1196-gen}.  Also, by decomposing $A = A_* \cup \bigcup_{j=0}^\infty 2^j \cdot A_j$ for odd primitive $A_j \subset \N_{\geq 1}$, and a remainder set $A_*$ consisting of at most one power of two, and then applying \eqref{nu-dilate}, we obtain the more complicated bound
$$ f(A \cap \N_{\geq k}({\mathcal Q})) \leq \sum_{j=0}^{k-1} 2^{-j} f( \N_{\geq k-j}(\mathcal{Q} \backslash \{2\}) ) + \frac{1}{2^k k \log 2}$$
in the general case.
\end{remark}

\section{\texorpdfstring{$2$}{2} is \Erdos{}-strong}\label{2-strong-sec}

In this section, we prove Theorem \ref{2-strong}.  We need to show that if $A$ is a primitive set of even numbers, then 
$$\sum_{n \in A} \nu_0(n) \leq \nu_0(2).$$
We can write $A = 2 \cdot A'$ for a primitive set $A'$ of natural numbers, and rewrite the objective as
\begin{equation}\label{n2}
\sum_{n \in A'} \nu_2(n) \leq \nu_2(1)
\end{equation}
where $\nu_2 \colon \N \to (0,\infty)$ is the rescaled weight
$$\nu_2(n) \coloneq 2 \nu_0(2n) = \frac{1}{n \log(2n)}.$$

Let $P(n \searrow m)$ be the von Mangoldt downward chain on $\N$ (with absorbing state $\{1\}$) from Example \ref{vmdc}.  From Lemma \ref{lem:mangoldt-tail-and-B}(iii) and \eqref{sub}, \eqref{vmdc-def} we see that $\nu_2$ is sub-invariant for this downward chain.  We then use Example \ref{adjoint} to form the adjoint upward Markov chain $P(n \nearrow m)$ on $\N \cup \{\infty\}$.

Let $b \colon \N \to [0,\infty)$ be the weight defined by setting 
$$ b(1) = \nu_2(1) = \frac{1}{\log 2}$$
and $b(n)=0$ for $n>1$.  From \eqref{h-recurse-upper}, \eqref{nu-recurse} and induction we thus have $h_{b\nearrow}(n) = \nu_2(n)$ for all $n \in \N$.  The claim \eqref{n2} now follows from \eqref{psum-weight-upper}.

\begin{remark} \label{remark:p-strong}
A modification of this argument can also be used to recover the result from \cite{Lichtman} that the odd primes $p$ are also \Erdos{}-strong, thus providing an alternate route to the proof of Conjecture \ref{conj:EPS}.
The basic setup uses the rescaled weight $\nu_p(n) \coloneq \frac{1}{n\log(pn)}$ and restricts the state space to $p$-rough integers $\mathcal R_p$.
(The set of $2$-rough integers $\mathcal R_2$ is simply all of $\N$.)
Instead of Lemma \ref{lem:mangoldt-tail-and-B}(iii), we need the following $p$-rough shifted estimate:
\begin{equation}\label{p-rough-shifted}
\sum_{\substack{q \in \mathcal R_p\\ q>1}}
\frac{\Lam(q)}{q\log(mq)\log(pmq)}
\leq
\frac{1}{\log(pm)}, \qquad \text{for $m \in \mathcal R_p$.}
\end{equation}
\end{remark}

\begin{remark} \label{remark:lean-erdos164}
All of the results of this section were formalized~\cite{lean-erdos164} in Lean by the first author using Codex from OpenAI, albeit in the flow language of Section~\ref{subsection:flow} instead of Markov chains.
This formalization includes two proofs of Conjecture~\ref{conj:EPS}: one using the ideas of Section~\ref{EPS-sec} and one using the ideas in this section.
\end{remark}

\section{Reproof of an inequality of Ahlswede, Khachatrian, and S\'ark\"ozy}\label{AKS-sec}

As a further quick application of the Markov chain machinery, we establish Theorem \ref{AKS-thm}.

Fix $3 \leq x \leq y$.  Without loss of generality, we may assume that the primitive set $A$ is contained in $[y/x,y]$.
Set $s \coloneq 1 - \frac{1}{10 \log x}$.  We consider the probability measure $w$ on $\N_{\geq 1}$ defined by setting
$$ w(p^j) \coloneq \frac{1/p^{js}}{Z}$$
if $p \leq x$ and $j \geq 1$, and $w(q)=0$ otherwise, where $Z$ is the ``partition function''
$$ Z \coloneq \sum_{p \leq x} \sum_{j=1}^\infty 1/p^{js}.$$
This is clearly a probability measure, and so we can form the upward Markov chain on $\N \cup \{\infty\}$ by using the simple multiplicative random walk defined in Example \ref{msrw}.  We then define an initial mass function $b \colon \N \to [0,+\infty)$ by setting $b(n) = \frac{1}{n^s}$ if $n \leq y$ and $n$ is \emph{$x$-rough} in the sense that all prime factors of $n$ are at least $x$, and $b(n)=0$ otherwise.  From \eqref{psum-weight-upper} we conclude that
\begin{equation}\label{p1}
 \sum_{n \in A} h_{b\nearrow}(n) \leq \sum_{n_0 \in {\mathcal N}} b(n_0).
 \end{equation}
From standard sieve theory (see, e.g., \cite[Chapter III.6, Theorem 3]{tenenbaum}) we see that the number of $x$-rough elements of $[2^k, 2^{k+1}]$ is $O(2^k / \log x)$ for any $k$, hence
\begin{equation}\label{p2}
\sum_{n_0 \in {\mathcal N}} b(n_0) \ll \sum_{2^k \ll y} \frac{2^k/\log x}{2^{ks}} \ll y^{1-s}
\end{equation}
by the choice of $s$ and the geometric series formula.  Now we lower bound the hitting probability $h_{b\nearrow}(n)$ for some $n \in [y/x,y]$.  We can factor
$$ n = m q_1 \dots q_k$$
where $m \leq y$ is $x$-rough, and $q_1,\dots,q_k$ are powers of primes less than or equal to $x$ that are pairwise coprime.  Each of the $k!$ permutations $q_{\pi(1)},\dots,q_{\pi(k)}$ of $q_1,\dots,q_k$ then creates an upward divisibility chain
$$ m | m q_{\pi(1)} | m q_{\pi(1)} q_{\pi(2)} | \dots | m q_{\pi(1)} \dots q_{\pi(k)} = n$$
from $m$ to $n$.  Calculating the probability of this chain arising as an upward chain from $m$, we obtain the lower bound
$$h_{b\nearrow}(n) \geq b(m) \frac{k! \prod_{i=1}^k 1/q_i^s}{Z^k} = \frac{1}{n^s} \frac{k!}{Z^k} = \frac{1}{n^s} \frac{1}{e^Z \mathbb{P}(X = k)}$$
where $X$ is a Poisson random variable of rate $Z$.  From the local central limit theorem (see, e.g., \cite[Chapter VII]{petrov}) or Stirling's formula, the maximum probability density $\sup_k \mathbb{P}(X=k)$ of such a variable is $O(1/\sqrt{Z})$, thus
$$h_{b\nearrow}(n) \gg \frac{\sqrt{Z}}{e^Z} \frac{1}{n^s},$$
so if we can establish the asymptotic
\begin{equation}\label{qasym}
 Z = \log\log x + O(1),
 \end{equation}
then we would have
$$h_{b\nearrow}(n) \gg \frac{\sqrt{\log\log x}}{\log x} \frac{1}{n^s} \asymp \frac{\sqrt{\log\log x}}{\log x} \frac{y^{1-s}}{n} $$
and the claim now follows from \eqref{p1}, \eqref{p2}.

It remains to show \eqref{qasym}.  The contribution of the higher powers $j \geq 2$ is $O( \sum_p \frac{1}{p^s (p^s-1)}) = O(1)$, so it suffices to show that
$$ \sum_{p \leq x} \frac{1}{p^s} = \log\log x + O(1).$$
But from the mean value theorem one has $\frac{1}{p^s} = \frac{1}{p} + O(\frac{\log p}{p \log x})$ for $p \leq x$, and the claim then follows from Theorem \ref{Mertens}.

\begin{remark}\label{lym-rem} The above argument also yields an LYM-type improvement to Theorem \ref{AKS-thm}, namely that
$$ \sum_{n \in A \cap [y/x,y]} \frac{1}{n \mathbb{P}(X=\omega_{\leq x}(n))} \ll \log x,$$
where $\omega_{\leq x}(n)$ is the number of prime factors of $n$ that are less than or equal to $x$, and $X$ is a Poisson random variable of rate $\log\log x+O(1)$.
\end{remark}

\section{Divisibility chains}\label{divis-sec}

In this section, we present a proof of Theorem \ref{conj:1217} using the Markov chain method.

Let $A$ be a set of natural numbers whose upper doubly logarithmic density
$$ \Delta \coloneq \limsup_{x\to\infty}\frac{1}{\log\log x} f(A \cap [1,x])$$
is positive (this is implied, for instance, if $A$ has positive lower logarithmic density).  In particular, we can find a sequence $x_j \to\infty$ such that
\begin{equation}\label{fax}
 f(A \cap [1,x_j]) = (\Delta-o(1)) \log\log x_j
 \end{equation}
as $j \to \infty$.

It will suffice to find a strictly increasing infinite divisibility chain
$$ n_1 | n_2 | n_3 | \dots$$
which has asymptotically large intersection with $A$ in the sense that
$$ \limsup_{j \to \infty} \frac{\# \{ i: n_i \in A \cap [1,x_j]\}}{\log\log x_j} \geq \Delta,$$
since the restriction of this divisibility chain to $A$ then gives the desired infinite chain in $A$.

By Example \ref{invariants}, the weight $\nuMangoldt$ is invariant for the von Mangoldt downward chain on $\N$ (with absorbing state $\{1\}$) from Example \ref{vmdc}.  Let $P(n \nearrow m)$ be the adjoint upward Markov chain on $\N \cup \{\infty\}$ given by Example \ref{adjoint}.  Then ${\mathbb P}_1$ generates an infinite divisibility chain
$$ 1 = n_0 | n_1 | n_2 | \dots,$$
which almost surely avoids the absorbing state $\infty$.  From \eqref{nu-recurse}, \eqref{h-recurse-upper} (starting from the Kronecker initial mass $b(n) = 1_{n=1}$) and induction we can compute the upward hitting probability
\begin{equation}\label{ksum}
\sum_{k=0}^\infty {\mathbb P}_1( n_k = n ) = \nuMangoldt(n)
\end{equation}
for all $n \in \N$. From \eqref{num-asym}, \eqref{fax} one has
$$ \sum_{n \in A \cap [1,x_j]} \nuMangoldt(n) = \sum_{n \in A \cap [1,x_j]} \nu_0(n)  + O(1) = (\Delta-o(1)) \log\log x_j$$
and hence
$$ \sum_{k=0}^\infty {\mathbb P}_1( n_k \in A \cap [1,x_j]) = (\Delta-o(1)) \log\log x_j.$$
Thus, if we introduce the random variables
$$ X_j \coloneq \frac{\# \{ i: n_i \in A \cap [1,x_j]\}}{\log\log x_j}$$
we see that
$$ \limsup_{j \to \infty} \E_1 X_j = \Delta.$$
We will shortly establish the uniform integrability bound
\begin{equation}\label{uniform}
 \sup_j \E_1 X_j^2 \ll 1.
\end{equation}
From the reverse Fatou lemma\footnote{Indeed, one can use the usual Fatou lemma applied to $N - \min(X_j,N)$ to obtain the truncated reverse Fatou inequality $\E_1 \limsup_{j \to \infty} \min(X_j,N) \geq \limsup_{j \to \infty} \E_1 \min(X_j,N)$, and use \eqref{uniform} to show that the right-hand side converges to $\limsup_{j \to \infty} \E_1 X_j$ as $N \to \infty$.}
(see, e.g., \cite[Corollary 4.2]{Kamihigashi2017}), this implies that
$$ \E_1 \limsup_{j \to \infty} X_j \geq \limsup_{j \to \infty} \E_1 X_j$$
and hence with positive probability one has
$$ \limsup_{j \to \infty} X_j \geq \Delta$$
which gives the claim.

It remains to prove \eqref{uniform}.  Expanding out the square, we can bound
$$ \E_1 X_j^2 \ll \frac{1}{(\log\log x_j)^2} \sum_{n \leq x_j} \sum_{1 \leq m \leq n} \mathbb{P}_1( m, n \in \{n_1, n_2, \dots\}).$$
If $n$ lies in the divisibility chain $\{n_1,n_2,\dots\}$, then the length of the divisibility chain prior to $n$ is at most $\Omega(n)$, and so for any given realization of that chain, only $\Omega(n)$ values of $m$ can occur.  Taking expectations, we obtain the conditional probability bound
$$ \sum_{1 \leq m \leq n} \mathbb{P}_1( m, n \in \{n_1, n_2, \dots\} \mid n \in \{n_1,n_2,\dots\}) \leq \Omega(n)$$
and hence by \eqref{ksum}
$$ \sum_{1 \leq m \leq n} \mathbb{P}_1( m, n \in \{n_1, n_2, \dots\}) \leq \Omega(n) \nuMangoldt(n).$$
From \eqref{num-asym} we have $\nuMangoldt(n) \ll \nu_0(n)$.
Summing in $n$ and using $\Omega(n) = \sum_p \sum_{j=1}^\infty 1_{p^j|n}$, we conclude that
$$  \E_1 X_j^2 \ll \frac{1}{(\log\log x_j)^2}\sum_p \sum_{j=1}^\infty \sum_{m \leq x / p^j} \nu_0(p^j m).$$
From \eqref{nu-dilate} one has
$$\sum_{m \leq x / p^j} \nu_0(p^j m) 
\leq \frac{1}{p^j} \sum_{m \leq x} \nu_0(m) \ll \frac{1}{p^j} \log\log x,$$
and the claim follows from Theorem \ref{Mertens}(ii).

\begin{remark}\label{gpt-rem-4} By taking contrapositives of \eqref{fah} we obtain the following ``single-scale'' version of Theorem \ref{conj:1217}: for any $2 \leq x \leq X$ and any non-empty $A \subset [x,\infty)$, one can find a strictly increasing divisibility chain $n_1 | \dots | n_h$ in $A$ of length at least
$$ h \geq \left( 1 - O\left(\frac{1}{\log x}\right) \right) f(A).$$  
As a consequence, if \eqref{fax} holds, then for each $j$ one can find a strictly increasing divisibility chain $n_{j,1} | \dots | n_{j,h_j}$ in $A \cap [1,x_j]$ with length at least
$$ h_j \geq (\Delta-o(1)) \log\log x_j$$
as $j \to \infty$.  Unfortunately, these chains are not related to each other as $j$ varies, and do not enjoy bounds uniform in $j$, so there is no obvious way to ``glue'' these finite divisibility chains into an \emph{infinite} strictly increasing divisibility chain to recover Theorem~\ref{conj:1217}.
\end{remark}

\section{Discussion}\label{discussion}

\subsection{Flow networks} \label{subsection:flow}

Many of the Markov chain arguments in this paper can be reformulated using the closely related language of flow networks.  We illustrate this with an alternate proof of Theorem \ref{conj:1196}.  Consider a directed weighted graph (or ``flow network'') on the natural numbers in which there is a directed edge from $nq$ to $n$ of weight 
$$w(nq \mapsto n) \coloneq \nu_0(nq) P(nq \searrow n) = \frac{\Lambda(q)}{nq \log^2(qn)}$$ 
whenever $n \geq 1$ and $q>1$, where $P(nq \searrow n)$ is the von Mangoldt downward chain from Example \ref{vmdc}.  For any $n > 1$, we can use Lemma \ref{lem:mangoldt-tail-and-B}(i) to compute the net inflow at $n$:
$$ \sum_q w(nq \mapsto n) = \frac{1}{n \log n} + O\left(\frac{1}{n \log^2 n}\right),$$
while from \eqref{vmi} we have an exact formula for the net outflow:
$$ \sum_{q|n} w(n \mapsto n/q) = \frac{1}{n \log n}.$$
Thus this flow network has a ``divergence'' of $O\left(\frac{1}{n \log^2 n}\right)$ at each $n \geq 1$.  

Now suppose that $A \subset [x,X]$ is primitive.  Let $\Omega$ be the set of all divisors of elements of $A$ which are greater than or equal to $x$.  Summing up all the divergences in $\Omega$, we see that the difference between the net inflow and net outflow of $\Omega$ is at most
$$ \sum_{n \geq x} O\left(\frac{1}{n \log^2 n}\right) = O\left(\frac{1}{\log x}\right),$$
thus
$$ \sum_{n \in \Omega, nq \notin \Omega} w(nq \mapsto n) = \sum_{n \in \Omega, q|n, n/q \notin \Omega} w(n \mapsto n/q) + O\left(\frac{1}{\log x}\right).$$
Every element $n$ of $A$ is maximal in $\Omega$ under the divisibility poset, and thus contributes $\frac{1}{n \log n} + O\left(\frac{1}{n \log^2 n}\right)$ to the inflow, thus
$$ \sum_{n \in \Omega, nq \notin \Omega} w(nq \mapsto n)  \geq f(A) -  O\left(\frac{1}{\log x}\right).$$
Meanwhile, the net outflow is at most
$$\sum_{n \in \Omega, q|n, n/q \notin \Omega} w(n \mapsto n/q) \leq \sum_{n \geq x > n/q, q|n} w(n \mapsto n/q)$$
which is bounded by $1 + O\left(\frac{1}{\log x}\right)$ thanks to \eqref{filip}.  This concludes the alternate proof of Theorem \ref{conj:1196}.  

\begin{remark}\label{gpt-rem-5} The above argument generalizes to a ``cut-capacity inequality'': if one has any downward Markov chain on a set ${\mathcal N}$ of natural numbers with absorbing states ${\mathcal A}$, $\nu$ is any sub-invariant weight for this chain, and $S$ is any subset of ${\mathcal N} \backslash {\mathcal A}$, then for any primitive set $A$ one has
$$ \sum_{n \in A \cap S} \nu(n) \leq \sum_{n \in S} \nu(n) \sum_{m|n: m \in {\mathcal N} \backslash S} P(n \searrow m).$$
For instance, one can recover Theorem~\ref{conj:1196} from this inequality by using the von Mangoldt downward flow, $S = [x,+\infty)$, and $\nu$ equal to either $\nu_0$ or $\nuMangoldt$.  We leave the proof of this inequality (and similar flow network reformulations of other arguments in this paper) to the interested reader.
\end{remark}

\subsection{Zeta process} \label{subsection:zeta}

The von Mangoldt measure $\nuMangoldt$ introduced in Example \ref{invariants} has an appealing probabilistic interpretation in terms of the following\footnote{This process was implicit in several proofs of Theorem \ref{conj:1196} that were generated by Arb Research using GPT-5.4; see \url{https://drive.google.com/file/d/1Ubj3t5QdkNu0u9WQipdaFQ8zln_JlbMJ/edit}.} ``zeta process''.  For each prime $p$ and natural number $k$, let $E_{p,k}$ be an independent exponential random variable with rate $\log p$, so that
$$ \mathbb{P}( E_{p,k} \geq s) = p^{-s}$$
for all $s>0$.  If we then let 
$$ e_{p,s} \coloneq \max \{ k : E_{p,1},\dots,E_{p,k} \geq s \}$$
denote the largest $k$ for which the $E_{p,k}$ consistently stay above $s$ for any prime $p$ and $s > 1$, then the $e_{p,s}$ are geometric random variables of mean $1/(p^s-1)$ which are independent in $p$, thus
$$ \mathbb{P}( e_{p,s} = k) = p^{-ks} (1-p^{-s})$$
for any $k \geq 0$.  As is well known, this implies that the product $Z_s \coloneq \prod_p p^{e_{p,s}}$ has the zeta distribution
\begin{equation}\label{zhit}
\mathbb{P}(Z_s = n) = \prod_p p^{-v_p(n)s} (1-p^{-s}) = \frac{1}{\zeta(s) n^s}.
\end{equation}
But as the $e_{p,s}$ are clearly non-increasing in $s$, we obtain a continuous divisibility chain that couples together the zeta distributions for $s>1$: $Z_s | Z_t$ whenever $1 < t < s$.

For $s > 1$ and $ds$ infinitesimal, the probability that the exponential random variable $E_{p,k}$ lies in $[s,s+ds]$ conditioning on $E_{p,k} \geq s$ is approximately
$$ -\frac{\dv{s}(p^{-s})}{p^{-s}}\,\dd s = \log p\,\dd s.$$
Summing over $p, k$ using \eqref{vmi}, one can conclude that the probability that there is a jump between $Z_s$ and $Z_{s+ds}$ conditioning on $Z_s$ is equal to $\log Z_s\,\dd s$ up to higher-order terms.  From this, one can show that the probability that a given natural number $n>1$ arises at least once in the continuous divisibility chain $Z_s$ (and which therefore will almost surely experience a transition between $Z_s = n$ and $Z_{s+ds} < n$ for some $s$) is
$$ \int_1^\infty \mathbb{P}(Z_s = n) \log n\,\dd s $$
which is precisely $\nuMangoldt(n)$ thanks to \eqref{zhit}, \eqref{mang-def}.

The weight $\nuMangoldt$ was already used in Section \ref{divis-sec} to prove Theorem \ref{conj:1217}.  As it is invariant rather than merely sub-invariant, it can be used as a replacement for $\nu_0$ in many of the other arguments in the paper, leading to some minor simplifications; for instance, one no longer needs to attach an absorbing state $\infty$ to the upward Markov chain.  We leave these modifications of the arguments to the interested reader.

\subsection{Further results} \label{subsection:further}

Przemek Chojecki\footnote{\url{https://www.erdosproblems.com/forum/thread/858}}, using GPT-5.4, has used similar methods to those in this paper to answer a separate question of \Erdos{} \cite[p.~128]{erdos70}, \cite[Problem \#858]{bloom}, namely determining the maximum value of
$$ \frac{1}{\log N} \sum_{n \in A} \frac{1}{n}$$
for a given large $N$, where $A \subset \{1,\dots,N\}$ is a set that does not have any solution to $at=b$ with $a,b \in A$ and with the smallest prime factor of $t$ at least $N$.  This maximum was determined to be $c+o(1)$ for a certain explicit constant $c = 0.618\dots$.

As observed by Will Sawin and Ofir Gorodetsky, respectively\footnote{\url{https://www.erdosproblems.com/forum/thread/1196}.}, analogues of these results can also be established in function field settings (where one considers the divisibility poset on polynomials over a finite field) or permutations (in which the partial order on permutations arises from deleting cycles and then relabeling).  

Several questions remain open in this subject. One such question, sometimes known as the \Erdos{} integer dilation approximation problem \cite{erdos61}, \cite{erdos73}, \cite{erdos77}, \cite[p.~101]{erdos-survey}, \cite{erdos92}, \cite{erdos97}, \cite[Problem \#143]{bloom} asks whether the \Erdos{} bound
$$ \sum_{x \in A} \frac{1}{x \log x} < \infty$$
continues to hold when $A$ is now a subset of $\R_{\ge2}$ which is ``approximately primitive'' in the sense that $|qx-y| \geq 1$ for all distinct $x,y \in A$ and natural numbers $q$.  This remains challenging with our techniques, as there does not seem to be a poset structure analogous to the divisibility poset with which to set up a useful Markov chain or flow network.  The weaker bound
$$ \sum_{x \in A \cap [1,N]} \frac{1}{x} = o(\log N)$$
as $N \to \infty$ was recently obtained in \cite{KKL} using the machinery of GCD graphs.

\section{Acknowledgments and AI disclosure}\label{acknowledgments-sec}

TT was supported by the James and Carol Collins Chair, the Mathematical Analysis \& Application Research Fund, and by NSF grant DMS-2347850, and is particularly grateful to recent donors to the Research Fund. 

ChatGPT was used to generate code for several of the images in this paper, to search for relevant literature (for instance, in locating references for the proof of Lemma \ref{lem:phi}), to proofread the paper and to offer additional suggested results and remarks\footnote{Specifically, Remarks \ref{GPT-rem-1}, \ref{GPT-rem-2}, \ref{GPT-rem-3}, \ref{gpt-rem-4}, and \ref{gpt-rem-5} were derived from comments suggested by GPT, although several further suggestions were discarded.}, and to perform numerics to guide the proof of Lemma \ref{lem:oddzeta}.  

The initial proof of Theorem \ref{conj:1196} was generated by an autonomous run\footnote{\url{https://chatgpt.com/share/69dd1c83-b164-8385-bf2e-8533e9baba9c}} of GPT-5.4 Pro; a similar run also established Theorem \ref{conj:1217}.  GPT-5.4 Pro was also used to assist with the initial proof of Theorem \ref{conj:EPS}, with the main human contributions being the downward divisor chain and suggesting Lemmas~\ref{lem:phi} and \ref{lem:mangoldt-tail-and-B}(ii) to establish the sub-invariance property. In addition, an early version of GPT-5.5 Pro was used to assist with the initial proof of Theorem \ref{conj:oddBM}. Finally, GPT-5.4 Pro helped prove Theorem~\ref{2-strong}. Nevertheless, the final proofs in this paper have been generated and reviewed by the human authors, using the AI-generated proofs as starting points when appropriate.

The Lean formalization in~\cite{lean-erdos164} was generated using OpenAI's Codex.
The Lean formalization in~\cite{lean-erdos1196} was generated using Math Inc.'s Gauss.

We thank Thomas Bloom for creating and moderating the website \cite{bloom}, where many of the key ideas and arguments presented in this paper were first shared and discussed by several members of the community around that website, including the authors and other contributors acknowledged separately in this paper.

\bibliographystyle{amsplain}

\begin{thebibliography}{99}

\bibitem{adell-lekuona}
J. A. Adell and A. Lekuona,
\emph{Dirichlet's eta and beta functions: Concavity and fast computation of their derivatives},
J. Number Theory \textbf{157} (2015), 215--222.

\bibitem{lean-erdos164}
B. Alexeev,
\textit{A Lean formalization of \Erdos Problem 164},
Lean 4 formalization, Mathlib v4.29.0 and Lean v4.29.0,
commit a9d31bcdffd1a68544b4e9214b867b2b34912fd2, 2026.
Available at
\url{https://github.com/plby/lean-proofs/blob/a9d31bcdffd1a68544b4e9214b867b2b34912fd2/ErdosProblems/Erdos164.md}.

\bibitem{alzer-kwong}
H. Alzer and M. K. Kwong,
\emph{On the concavity of Dirichlet's eta function and related functional inequalities},
J. Number Theory \textbf{151} (2015), 172--196.

\bibitem{AKS}
R. Ahlswede, L. Khachatrian, and A. S\'ark\"ozy,
\emph{On the density of primitive sets},
J. Number Theory \textbf{108} (2004), 319--361.

\bibitem{banks-martin}
W. D. Banks and G. Martin,
\emph{Optimal primitive sets with restricted primes},
Integers \textbf{13} (2013), Paper No. A69, 10 pp.

\bibitem{behrend}
F. Behrend,
\emph{On sequences of numbers not divisible one by another},
J. London Math. Soc. \textbf{10} (1935), 42--44.

\bibitem{bloom}
T. F. Bloom,
\emph{\Erdos{} problems},
\url{https://www.erdosproblems.com/}.

\bibitem{cohen}
H. Cohen, Number Theory, Volume II: Analytic and Modern Tools, GTM Vol. 240, Springer, 2007.

\bibitem{davenport}
H. Davenport and P. \Erdos{},
\emph{On sequences of positive integers},
Acta Arith. \textbf{2} (1936), 147--151.

\bibitem{erdos35}
P. \Erdos{},
\emph{Note on sequences of integers no one of which is divisible by any other},
J. London Math. Soc. \textbf{10} (1935), 126--128.

\bibitem{erdos61}
P. \Erdos{},
\emph{Some unsolved problems}, Magyar Tud. Akad. Mat. Kutat\'o Int. K\"ozl. (1961), 221--254.

\bibitem{erdos70}
P. \Erdos{},
\emph{Some extremal problems in combinatorial number theory}, Mathematical Essays Dedicated to A. J. Macintyre (1970), 123--133.

\bibitem{erdos73}
P. \Erdos{},
\emph{Problems and results on combinatorial number theory}, A survey of combinatorial theory (Proc. Internat. Sympos., Colorado State Univ., Fort Collins, Colo., 1971) (1973), 117--138.

\bibitem{erdos76}
P. \Erdos{}, Conjecture 2.1, 
\emph{Problems and results on combinatorial number theory. II},
J. Indian Math. Soc. (N.S.) (1976), 285--298.

\bibitem{erdos77}
P. \Erdos{},
\emph{Problems and results on combinatorial number theory. III}, Number theory day (Proc. Conf., Rockefeller Univ.,
New York, 1976) (1977), 43--72.

\bibitem{erdos-survey}
P. \Erdos{},
\emph{A survey of problems in combinatorial number theory},
Ann. Discrete Math. (1980), 89--115.

\bibitem{erdos86}
P. \Erdos{},
\emph{Probl\'emes et r\'esultats en th\'eorie des nombres},
Conf\'erence de Paul \Erdos \'a l'Universit\'e de Limoges le 21 Octobre 1986, r\'edig\'e par Fran\c{c}ois Morain.

\bibitem{erdos92}
P. \Erdos{},
\emph{Some of my forgotten problems in number theory}, Hardy-Ramanujan J. (1992), 34--50.

\bibitem{erdos97}
P. \Erdos{},
\emph{Some of my favorite problems and results}, The mathematics of Paul \Erdos{}, I (1997), 47--67.

\bibitem{ess0}
P. \Erdos{}, A. S\'ark\"ozy, and E. Szemer\'edi,
\emph{On divisibility properties of sequences of integers}, Studia Sci. Math. Hungar. \textbf{1} (1966), 431--435.

\bibitem{ess1}
P. \Erdos{}, A. S\'ark\"ozy, and E. Szemer\'edi,
\emph{On an extremal problem concerning primitive sequences},
J. London Math. Soc. \textbf{42} (1967), 484--488.

\bibitem{ess2}
P. \Erdos{}, A. S\'ark\"ozy, and E. Szemer\'edi,
\emph{On divisibility properties of sequences of integers},
Colloq. Math. Soc. J\'anos Bolyai \textbf{2} (1970), 35--49.

\bibitem{ezhang} P. Erd\H os, Z. Zhang, \textit{Upper bound of $\sum 1/(a_i \log a_i)$ for primitive sequences}, Proc. Amer. Math. Soc., {\bf 117} (1993), 891--895.

\bibitem{glw}
O. Gorodetsky, J. D. Lichtman, and M. D. Wong,
\emph{On \Erdos{} sums of almost primes},
Comptes Rendus. Math\'ematique \textbf{362} (2024), no. G12, 1571--1596.

\bibitem{HalbRoth} H. Halberstam, K. F. Roth, {\it Sequences}, Oxford University Press, Oxford (1966)

\bibitem{Hsetmult} R. R. Hall, {\it Sets of multiples}, Cambridge Tracts in Mathematics, {\bf 118}, Cambridge University Press, Cambridge (1996)

\bibitem{Kamihigashi2017}
T. Kamihigashi,
\emph{A generalization of Fatou's lemma for extended real-valued functions on $\sigma$-finite measure spaces: With an application to infinite-horizon optimization in discrete time},
J. Inequal. Appl. \textbf{2017} (2017), Paper No. 24, 15 pp.

\bibitem{KKL}
D. Koukoulopoulos, Y. Lamzouri, and J. D. Lichtman, \emph{\Erdos{}'s integer dilation approximation problem and GCD graphs}, arXiv:2502.09539.


\bibitem{Lichtman}
J. D. Lichtman,
\emph{A proof of the \Erdos{} primitive set conjecture},
Forum Math. Pi \textbf{11} (2023), Paper No. e18, 21 pp.; arXiv:2202.02384.

\bibitem{Lalmost}
J. D. Lichtman,
\emph{Almost primes and the Banks--Martin conjecture},
J. Number Theory \textbf{211} (2020), 513--529; arXiv:1909.00804.

\bibitem{LP}
J. D. Lichtman and C. Pomerance,
\emph{The \Erdos{} conjecture for primitive sets},
Proc. Amer. Math. Soc. Ser. B \textbf{6} (2019), 1--14.

\bibitem{lubell}
D. Lubell, \emph{A short proof of Sperner's lemma}, Journal of Combinatorial Theory \textbf{1} (1966), 299.

\bibitem{lean-erdos1196}
Math Inc.,
\textit{Primitive Sets Above \(x\) in Lean},
Lean 4 formalization,
Lean v4.30.0-rc1,
commit 02fba13be7487cc51315f68d8fa7ef277633d3c8, 2026.
Available at
\url{https://github.com/math-inc/Erdos1196/tree/02fba13be7487cc51315f68d8fa7ef277633d3c8}.

\bibitem{meshalkin}
L. D. Meshalkin, \emph{Generalization of Sperner's theorem on the number of subsets of a finite set}, Theory of Probability and Its Applications, \textbf{8} (1963), 203--204.

\bibitem{montgomeryvaughan}
H. L. Montgomery and R. C. Vaughan,
\emph{Multiplicative Number Theory I: Classical Theory},
Cambridge Studies in Advanced Mathematics, vol. 97,
Cambridge University Press, 2007.

\bibitem{petrov}
V. V. Petrov, Sums of Independent Random Variables, Springer, 1975.

\bibitem{sarkozy1}
A. S\'ark\"ozy, \emph{On divisibility properties of sequences of integers}, 241--250, in The Mathematics of Paul \Erdos{} I, R. L. Graham, J. Nesetril (eds.), Springer-Verlag Berlin Heidelberg (1997). 

\bibitem{sarkozy2}
A. S\'ark\"ozy, \emph{On divisibility properties of sequences of integers}, 221--232, in The Mathematics of Paul Erd\H{o}s I (2nd edition) R. L. Graham, J. Nesetril (eds.), Springer-Verlag Berlin Heidelberg (2013).

\bibitem{sperner}
E. Sperner, \emph{Ein Satz \"uber Untermengen einer endlichen Menge}, Mathematische Zeitschrift \textbf{27} (1928), 544--548.

\bibitem{stanley}
R. Stanley, \emph{Two poset polytopes},
Discrete Comput. Geom. 1 (1986), no. 1, 9--23.

\bibitem{tenenbaum}
G. Tenenbaum,
\emph{Introduction to Analytic and Probabilistic Number Theory},
Graduate Studies in Mathematics, vol. 163,
American Mathematical Society, Providence, RI, 2015.

\bibitem{vdl}
J. van de Lune,
\emph{Some inequalities involving Riemann's zeta-function},
CWI Report ZW 50/75, Centrum Wiskunde \& Informatica, Amsterdam, 1975.

\bibitem{various}
Various,
\emph{Some of Paul's favorite problems},
booklet produced for the conference ``Paul \Erdos{} and his mathematics,''
Budapest, July 1999.

\bibitem{wang}
K. C. Wang,
\emph{The logarithmic concavity of $(1-2^{1-r})\zeta(r)$},
J. Changsha Comm. Univ. \textbf{14} (1998), 1--5. In Chinese.

\bibitem{yamamoto}
K. Yamamoto, \emph{Logarithmic order of free distributive lattice}, Journal of the Mathematical Society of Japan \textbf{6} (1954), 343--353.

\bibitem{zhang1} Z. Zhang, \textit{On a conjecture of Erd\H os on the sum $\sum_{p\le n} 1/(p \log p)$}, J. Number Theory {\bf 39} (1991), 14--17.

\bibitem{zhang2} Z. Zhang, \textit{On a problem of Erd\H os concerning primitive sequences}, Math. Comp.\ {\bf 60} (1993), 827--834.


\end{thebibliography}

\end{document}